\documentclass[11pt]{article}
\def\arxiv{}

\usepackage[
  paper  = letterpaper,
  left   = 1.2in,
  right  = 1.2in,
  top    = 1in,
  bottom = 1in,
  ]{geometry}

\usepackage{amsmath,amssymb,amsthm,mathtools,mathrsfs,braket,bbm,dsfont}
\usepackage{graphicx,subcaption,tikz}
\usepackage{enumitem}
\usepackage{algorithm,algcompatible}
\usepackage{adjustbox}

\usepackage{authblk}
\usepackage{booktabs}
\usepackage[numbers,sort&compress]{natbib}
\usepackage[hidelinks]{hyperref}

\usepackage[T1]{fontenc}
\usepackage[only,llbracket,rrbracket]{stmaryrd}

\newtheorem{theorem}{Theorem}
\newtheorem{lemma}[theorem]{Lemma}

\newtheorem{remark}[theorem]{Remark}
\newtheorem{corollary}[theorem]{Corollary}
\newtheorem{definition}[theorem]{Definition}

\newtheorem{assumption}{Assumption}

\newcommand{\acks}[1]{\section*{Acknowledgements}#1}

\title{Fast Minimization of Expected Logarithmic Loss\\ via Stochastic Dual Averaging}

\author[1]{Chung-En Tsai}
\author[2,3,4,5,6]{Hao-Chung Cheng}
\author[1,3,4]{Yen-Huan Li}
\affil[1]{Department of Computer Science and Information Engineering,\protect\\National Taiwan University}
\affil[2]{Department of Electrical Engineering and Graduate Institute of\protect\\Communication Engineering, National Taiwan University}
\affil[3]{Department of Mathematics, National Taiwan University}
\affil[4]{Center for Quantum Science and Engineering,\protect\\National Taiwan University}
\affil[5]{Physics Division, National Center for Theoretical Sciences}
\affil[6]{Hon Hai (Foxconn) Quantum Computing Centre}

\date{\vspace{-5ex}}

\newcommand{\R}{\mathbb{R}}
\newcommand{\N}{\mathbb{N}}

\newcommand{\C}{\mathbb{C}}

\renewcommand{\H}{\mathbb{H}}

\newcommand{\D}{\mathscr{D}}
\newcommand{\E}{\mathbb{E}}
\newcommand{\V}{\mathbb{V}}
\renewcommand{\O}{\mathcal{O}}

\DeclareMathOperator{\ri}{ri}
\DeclareMathOperator{\dom}{dom}

\DeclareMathOperator{\per}{per}
\DeclareMathOperator{\rel}{rel}

\DeclareMathOperator*{\argmin}{arg\!\min}
\DeclarePairedDelimiter\abs{\lvert}{\rvert}%
\DeclarePairedDelimiter\norm{\lVert}{\rVert}%
\DeclarePairedDelimiter\brackets{\llbracket}{\rrbracket}
\DeclareMathOperator*{\tr}{tr}

\newcommand{\du}{\mathrm{d}}
\newcommand{\tildeO}{\smash{\tilde{O}}}

\begin{document}

\maketitle

\begin{abstract}
% !TEX root = ../paper.tex
Consider the problem of minimizing an expected logarithmic loss over either the probability simplex or the set of quantum density matrices.
This problem includes tasks such as solving the Poisson inverse problem, computing the maximum-likelihood estimate for quantum state tomography, and approximating positive semi-definite matrix permanents with the currently tightest approximation ratio. 
Although the optimization problem is convex, standard iteration complexity guarantees for first-order methods do not directly apply due to the absence of Lipschitz continuity and smoothness in the loss function.

In this work, we propose a stochastic first-order algorithm named $B$-sample stochastic dual averaging with the logarithmic barrier.
For the Poisson inverse problem, our algorithm attains an $\varepsilon$-optimal solution in $\smash{\tilde{O}}(d^2/\varepsilon^2)$ time, matching the state of the art, where $d$ denotes the dimension.
When computing the maximum-likelihood estimate for quantum state tomography, our algorithm yields an $\varepsilon$-optimal solution in $\smash{\tilde{O}}(d^3/\varepsilon^2)$ time.
This improves on the time complexities of existing stochastic first-order methods by a factor of $d^{\omega-2}$ and those of batch methods by a factor of $d^2$, where $\omega$ denotes the matrix multiplication exponent.
Numerical experiments demonstrate that empirically, our algorithm outperforms existing methods with explicit complexity guarantees.
\end{abstract}

% !TEX root = ../paper.tex

\ifdefined\arxiv
\section{Introduction}
\else
\section{INTRODUCTION}
\fi
\label{sec:introduction}
Denote by $\D_d$ the set of $d \times d$ {quantum density matrices}, i.e., the set of Hermitian positive semi-definite (PSD) matrices of unit traces. 
Let $P$ be a probability distribution over the set of $d \times d$ Hermitian PSD matrices. 
Consider the optimization problem of minimizing an expected logarithmic loss:
\begin{equation} \label{eq:quantum_problem}
	f^\star = \min_{\rho\in\D_d} f(\rho),\quad f(\rho)\coloneqq\E_{A\sim P}[-\log\tr(A\rho)] . % ,
\end{equation}
% where $\D_d$ is the set of $d\times d$ \textit{quantum density matrices}, i.e., Hermitian positive semi-definite (PSD) matrices of unit traces, and $P$ is a distribution over the set of $d\times d$ Hermitian PSD matrices.
A point $\hat{\rho}\in\D_d$ is said to be $\varepsilon$-optimal if $f(\hat{\rho}) - f^\star \leq \varepsilon$.
When both $A$ and $\rho$ are restricted to diagonal matrices, the optimization problem \eqref{eq:quantum_problem} reduces to
\begin{equation} \label{eq:classical_problem}
	\min_{x\in\Delta_d} f(x),\quad f(x)\coloneqq\E_{a\sim P'}[-\log\braket{a,x}],
\end{equation}
%when both $A$ and $\rho$ are restricted to diagonal matrices,
where $\Delta_d$ denotes the probability simplex in $\R^d$ and $P'$ is a probability distribution over $[0,\infty)^d$.
We refer to the two problems \eqref{eq:quantum_problem} and \eqref{eq:classical_problem} as the \textit{quantum} setup and the \textit{classical} setup, respectively.
For the special cases when
\[
\begin{split}
	f(\rho)&\coloneqq\frac{1}{n}\sum_{i=1}^n -\log\tr(A_i\rho),\ \text{or}\\
	f(x)&\coloneqq\frac{1}{n}\sum_{i=1}^n -\log\braket{a_i,x},
\end{split}
\]
we say the optimization problems are in the \textit{finite-sum} setting with sample size $n$.

Notably, the classical setup \eqref{eq:classical_problem} encompasses the problem of computing Kelly's criterion, an asymptotically optimal strategy in long-term investment \citep{Kelly1956a,Algoet1988a,MacLean2011a}.
It is also equivalent to solving the Poisson inverse problem, which finds applications in positron emission tomography (PET) in medical imaging and astronomical image denoising \citep{Vardi1998a,Bertero2018a}.
%Regarding PET, the convergence of ordered subset expectation maximization (OSEM) \citep{Hudson1994a}, a widely used method, has been open for more than twenty years.
Lastly, the problem is the batch counterpart of the online portfolio selection problem \citep{Cover1991a}, for which designing algorithms that are optimal in both regret and computational complexity has remained unsolved for over thirty years \citep{van-Erven2020a,Jezequel2022a}.

The quantum setup \eqref{eq:quantum_problem} has even broader applications.
One important example is computing the maximum-likelihood (ML) estimate for quantum state tomography \citep{Hradil1997a}, a fundamental task for the verification of quantum devices.
%Computing the estimate is computationally expensive and has been reported to take weeks of computation \citep{Gross2010a}.
Another example is computing a semidefinite programming relaxation of the PSD matrix permanents \citep{Yuan2022a}.
The relaxation achieves the currently tightest approximation ratio and can be used to estimate the output probabilities of Boson sampling experiments \citep{Aaronson2011a}.

%In this work, we focus on designing efficient algorithms for solving the problem (\ref{eq:quantum_problem}).
Although the optimization problems \eqref{eq:quantum_problem} and \eqref{eq:classical_problem} are convex, standard convex optimization methods face two challenges.
The first challenge is the lack of Lipschitz continuity and smoothness in the loss function \citep{Li2019a}.
As a result, iteration complexity guarantees of standard first-order methods, such as mirror descent and dual averaging, do not directly apply \citep{Nesterov2018a,Lan2020a}.
While second-order methods, such as Newton's method, do possess explicit complexity guarantees \citep{Nesterov2018a}, they still face the second challenge.

The second challenge is the scalabilities with respect to the dimension $d$ and sample size $n$ in the finite-sum setting.
For instance, the dimension $d$ typically reaches millions, and the sample size $n$ can exceed hundreds of millions in PET \citep{Ben-Tal2001a,Ehrhardt2017a}.
Both $d$ and $n$ grow exponentially with the number of qubits (quantum bits) in quantum state tomography \citep{Chen2023a}.
However, the per-iteration time complexities of batch methods 
grow at least linearly with $n$, and those of second-order methods scale poorly with $d$ as they require computing Hessian inverses.
% due to the inversions of Hessian matrices.
% Consequently, batch methods and second-order methods are computationally inefficient in practice.

%the time complexities of batch methods and second-order methods scale poorly with respect to $n$ and $d$, respectively, making them computationally inefficient in practice.

%Regarding the issues of high dimensionality and large sample size,
When dealing with high dimensionality and large sample sizes, stochastic first-order algorithms, such as stochastic gradient descent, are preferred.
Their per-iteration time complexities can be independent of the sample size $n$, and they do not require computationally demanding operations involving Hessian matrices.
Nevertheless, standard stochastic first-order algorithms continue to face the challenge related to the lack of Lipschitz continuity and smoothness, as mentioned earlier.

Due to the absence of Lipschitz continuity and smoothness, mini-batch stochastic Q-Soft-Bayes (SQSB) \citep{Lin2021a,Lin2022a} and stochastic Q-LB-OMD (SQLBOMD) \citep{Tsai2022a,Hanzely2021a} are the only two stochastic first-order methods with clear time complexity guarantees for solving the problems \eqref{eq:quantum_problem} and \eqref{eq:classical_problem}.
% for solving the problems.
Both algorithms do not compete with batch methods in terms of the empirical convergence speed, as shown in Section \ref{sec:numerical_results}.
Stochastic mirror descent studied by \citet{DOrazio2021a} is only guaranteed to solve the problems up to an arbitrarily large error\footnote{\citet{DOrazio2021a} proved that the average Bregman divergence to the minimizer asymptotically converges to $\sigma_{\mathcal{X}}^2$, which equals $f^\star$ when, for example, $\norm{ a_i }_\infty = 1$ in the classical setup and can be arbitrarily large in general.}. 
Other stochastic first-order methods, such as stochastic primal-dual hybrid gradient (SPDHG) \citep{Chambolle2018a,Alacaoglu2022a}, stochastic mirror-prox \citep{He2020a}, and stochastic coordinate descent \citep{Fercoq2019a}, are only guaranteed to converge asymptotically.
%Lastly, existing stochastic mirror descent methods based on relative smoothness \citep{Hanzely2021a,DOrazio2021a} do not directly apply, though the loss function is actually smooth relative to the logarithmic barrier \eqref{eq:log_barrier} \citep{Tsai2023a}.
%The logarithmic barrier violates the relative strong convexity condition required by those existing results.

%This works aims at designing empirically fast stochastic first-order methods with explicit complexity guarantees.

\paragraph{Contributions}

In this work, we propose a mini-batch stochastic first-order algorithm named $B$-sample stochastic dual averaging with the logarithmic barrier (LB-SDA,~Algorithm~\ref{alg:sda}) for solving the optimization problems \eqref{eq:quantum_problem} and \eqref{eq:classical_problem}, where $B$ denotes the mini-batch size.
The expected optimization error of $B$-sample LB-SDA vanishes at a rate of $\tildeO ( d / t + \sqrt{ d / ( B t ) } )$.
This matches the standard results of mini-batch stochastic gradient descent for minimizing \emph{smooth} functions \citep{Dekel2012a}, regardless of the absence of smoothness in our problem.
%The per-iteration time of $B$-sample LB-SDA is $\tildeO(Bd)$ and $\tildeO ( Bd^2 + d ^ \omega )$ in the classical setup and the quantum setup, respectively, where $\omega \in [2,2.372)$ denotes the matrix multiplication exponent \citep{Williams2023a}.

In the classical setup, the time complexity of obtaining an $\varepsilon$-optimal solution via $B$-sample LB-SDA is $\tildeO(d^2/\varepsilon^2)$, matching the state of the art of stochastic first-order methods \citep{Lin2022a,Tsai2022a}.
In the quantum setup, the time complexity of obtaining an $\varepsilon$-optimal solution is $\tildeO(d^3/\varepsilon^2)$ when the mini-batch size is set to $d$.
This improves the dimension dependence of existing stochastic algorithms by a factor of $d^{\omega-2}$, where $\omega \in [2,2.372)$ denotes the matrix multiplication exponent \citep{Williams2024a}.
It is worth noting that in practical implementation, such as BLAS \citep{Dongarra1990a}, $\omega$ is effectively $3$.
The time complexity guarantee also improves that of the currently fastest batch algorithm by a factor of $n/d$.
Such improvement is significant, given that $n=\Omega(d^3)$ is necessary for ML quantum state tomography \citep{Chen2023a}.

%Therefore in this work, we introduce a novel smoothness characterization of the logarithmic loss (Lemma~\ref{lem:self_bounding}), which generalizes the one of \citet{Tsai2023b} to the quantum setup and greatly simplifies the proof therein.
%Standard analysis for mini-batch stochastic gradient descent requires smoothness of the loss function.
%We overcome the lack of smoothness by introducing a novel smoothness characterization of the logarithmic loss (Lemma~\ref{lem:self_bounding}), which generalizes the one of \citet{Tsai2023b} to the quantum setup and greatly simplifies the proof therein.

\aboverulesep=0ex
\belowrulesep=0ex
\renewcommand{\arraystretch}{1.5}
\begin{table*}[t!]
\adjustbox{max width=\textwidth}{%
\centering
\begin{tabular}{l|c|c|c|l}
\toprule
Algorithms & Iter. complexity & Per-iter. time & Time complexity & References \\
\midrule
NoLips & $d/\varepsilon$ & $nd^2 + d^\omega$ & $(nd^{3} + d^{\omega+1})/\varepsilon$ & \citet{Bauschke2017a} \\
\hline
QEM & $\log d/\varepsilon$ & $nd^2+d^\omega$ & $(nd^2+d^\omega)/\varepsilon$ & \citet{Lin2021a}\\
\hline
Frank-Wolfe & $n/\varepsilon$ & $nd^\omega$ & $n^2d^\omega/\varepsilon$ &  \citet{Zhao2023a} \\
\hline
SQSB & $d/\varepsilon^2$ & $d^\omega$ & $d^{\omega+1}/\varepsilon^2$ & \citet{Lin2022a}\\
\hline
SQLBOMD & $d/\varepsilon^2$ & $d^\omega$ & $d^{\omega+1}/\varepsilon^2$ & \citet{Tsai2022a}\\
\hline
\bf $1$-sample LB-SDA & $d/\varepsilon^2$ & $d^\omega$ & $d^{\omega+1}/\varepsilon^2$ & \bf Corollary~\ref{cor:sda} (this work) \\
\hline
\bf $d$-sample LB-SDA & $1/\varepsilon^2$ & $d^3$ & $d^{3}/\varepsilon^2$ & \bf Corollary~\ref{cor:sda} (this work) \\
\hline
\bf $B$-sample LB-SDA & $d/(B\varepsilon^2)$ & $Bd^2 + d^\omega$ & $d^3/\varepsilon^2 + d^{\omega+1}/(B\varepsilon^2)$ & \bf Corollary~\ref{cor:sda} (this work) \\
\bottomrule
\end{tabular}}
\caption{A comparison of existing first-order methods for the quantum setup \eqref{eq:quantum_problem} with explicit complexity guarantees.
Iteration complexity and time complexity refer to the number of iterations and arithmetic operations required to obtain an $\varepsilon$-optimal solution, respectively.
We assume $t\gg d^2$ and omit logarithmic factors, where $t$ denotes the number of iterations.}
\label{tab:time_complexity_comparison}
\end{table*}

Lastly, we conducted numerical experiments to demonstrate the efficiency of the proposed method.
The numerical results suggest that $1$-sample LB-SDA is the currently fastest method with explicit complexity guarantees for the Poisson inverse problem, and $d$-sample LB-SDA outperforms all methods in terms of fidelity, a standard measure of the closeness of quantum states, for ML quantum state tomography.
To the best of our knowledge, this is the first empirical evidence that stochastic first-order algorithms can surpass batch ones in computing the ML estimate for quantum state tomography.

\paragraph{Technical Breakthroughs}

% The time complexity improvement is based on a novel analysis of mini-batch algorithms.
% The 
Our analysis 
% combines 
consists of three key ingredients: a regret bound of \citet{Tsai2023b}, a smoothness characterization of the logarithmic loss (Lemma~\ref{lem:self_bounding}), and a new local-norm-based analysis of the standard online-to-batch conversion \citep{Cesa-Bianchi2004a}, all of which are of independent interest.

% Informally, t
\citet{Tsai2023b} proved the following regret bound for online convex optimization with the logarithmic loss on the probability simplex $\Delta_d$ (Appendix~\ref{app:olo}):
\begin{equation}
R_t \leq \tildeO \left(\sqrt{d\sum_{\tau=1}^t \norm{\nabla f_\tau(\rho_\tau)}_{\rho_\tau,\ast}^2 }\right) \leq \tildeO\left(\sqrt{dt}\right), \label{eq_regret_bound}
\end{equation}
where $\norm{\cdot}_{\rho,\ast}$ is the dual local norm associated with the logarithmic barrier.
% It's important to note that d
Directly applying the standard online-to-batch conversion \citep{Cesa-Bianchi2004a} with the second upper bound
% with the right upper bound 
can only yield an optimization error bound of $\tildeO(\sqrt{d/t})$, which is independent of the mini-batch size. 
% and does not lead to the time complexity improvement.
In comparison, we make use of the finer first upper bound and derive an optimization error bound of $\tilde{O} ( d / t + \sqrt{ d / ( Bt ) } )$. 
This leads to a time complexity bound of $\tildeO(d^3/\varepsilon^2 + d^{\omega+1}/(B\varepsilon^2))$ for the quantum setup, which creates space for improved dimensional scalability via choosing the mini-batch size $B$. 
See Section \ref{sec:time_complexity} for a detailed discussion. 

To make use of the finer first regret bound \eqref{eq_regret_bound}, we generalize the smoothness characterization of the logarithmic loss of \citet{Tsai2023b} for the quantum setup.
The original proof of \citet{Tsai2023b} is 
% complicated and 
challenging to generalize due to the noncommutativity in the quantum setup.
Our generalization is based on a great simplification of their proof by utilizing self-concordance properties of the logarithmic loss (Appendix~\ref{app:self_concordance_relative_smoothness}). 
% We greatly simplify their proof by utilizing self-concordance properties of the logarithmic loss (Appendix~\ref{app:self_concordance_relative_smoothness}), and extend the simplified proof to the quantum setup.

Our analysis modifies that of the anytime online-to-batch conversion \citep{Cutkosky2019a} to handle the local norms.
%The analysis is essentially a local-norm extension of the analysis of the anytime online-to-batch conversion \citep{Cutkosky2019a}.
It is worth noting that we also adapt the analysis of anytime online-to-batch for the standard one, as the latter has shown better empirical performance.
% Apart from the analysis, we choose the standard online-to-batch conversion instead of the anytime version because the former appears to perform better empirically.

%To make use of the finer upper bound in the middle, we incorporate the smoothness characterization of the logarithmic loss into our analysis of the online-to-batch conversion.
%The smoothness characterization is used to bound the local norm by the suboptimality in function values.
%Our analysis modifies that of the anytime online-to-batch conversion \citep{Cutkosky2019a} to handle the local norms.
%%The analysis is essentially a local-norm extension of the analysis of the anytime online-to-batch conversion \citep{Cutkosky2019a}.
%Apart from the analysis, we choose the standard online-to-batch conversion instead of the anytime version because the former appears to perform better empirically.

%The smoothness characterization is crucial since a direct application of the online-to-batch conversion without it can only yield an optimization error bound of $\tildeO(\sqrt{d/t})$.
%Such bound is independent of the batch size and does not lead to the improvement.

%The classical counterpart of the smoothness characterization has appeared in the work of \citet{Tsai2023b}.
%Their proof is complicated and challenging to generalize due to the noncommutativity in the quantum setup.
%We greatly simplify their proof by utilizing self-concordance properties of the logarithmic loss (Appendix~\ref{app:self_concordance_relative_smoothness}), and extend the simplified proof to the quantum setup.

\paragraph{Notations}
We denote the set $\{1,2,\ldots,n\}$ by $\brackets{n}$ for a natural number $n\in\N$.
We denote the $\ell_p$-norm by $\norm{\cdot}_p$ for $p\in [1,\infty]$.
%We denote the set of nonnegative real numbers by $\R_+$.
We denote the sets of $d\times d$ Hermitian matrices, Hermitian PSD matrices, and Hermitian positive definite matrices by $\H^d$, $\H_+^d$, and $\H_{++}^d$, respectively.
We denote the relative interior of a set $S$ by $\ri S$.
We denote the $i$-th entry of a vector $v$ by $v(i)$.
We denote the conjugate transpose of a matrix $U$ by $U^\ast$.
We denote the sum of time-indexed matrices $A_1, \ldots, A_t \in \H^d$ by $A_{1: t}$.
We define the domain of a function $f:\H^d\to\R\cup\{ \infty \}$ by $\smash{\dom f\coloneqq\{ \rho\in\H^d \mid f(\rho)<\infty \}}$.
%The asymptotic notation $\tildeO$ omits logarithmic factors. 

% !TEX root = ../paper.tex
\ifdefined\arxiv
\section{Related Work} 
\else
\section{RELATED WORK}
\fi
\label{sec:related_work}

%\subsection{Optimization Algorithms}
The relationships between this work and the works of \citet{Tsai2023b} and \citet{Cutkosky2019a} have been addressed in Section~\ref{sec:introduction}.
This section focuses on optimization algorithms.

Although standard optimization methods are not suitable for solving the problems \eqref{eq:quantum_problem} and \eqref{eq:classical_problem}, several methods with clear complexity guarantees have been proposed in the last decade.
Table~\ref{tab:time_complexity_comparison} summarizes existing results.
We focus on batch methods below as stochastic methods have already been discussed in Section~\ref{sec:introduction}.

QEM \citep{Lin2021a} is the current theoretically fastest batch method that solves the optimization problem \eqref{eq:quantum_problem} with clear complexity guarantees.
Its classical counterpart EM was proposed by \citet{Shepp1982a} and \citet{Cover1984a} independently.
%QEM \citep{Lin2021a} is the currently theoretically fastest batch method for solving the optimization problem \eqref{eq:quantum_problem}, generalizing its classical counterpart proposed by \citet{Shepp1982a} and \citet{Cover1984a}.
%The optimization error of QEM vanishes at a rate of $O(\log d/t)$ for both the classical and the quantum setup \citep{Iusem1992a,Lin2021a}.
%The convergence rate of $O(\log d/t)$ was established in the classical setup and the quantum setup by \citet{Iusem1992a} and \citet{Lin2021a}, respectively.
%Given our loss function is smooth relative to the logarithmic barrier \eqref{eq:log_barrier}, Bregman proximal gradient (BPG) \citep{Bauschke2017a} has a clear complexity guarantee\footnote{Existing results for stochastic mirror descent \citep{Hanzely2021a,DOrazio2021a} require a strongly convex reference function.
%However, the logarithmic barrier does not satisfy the condition.}.
%Recently, \citet{Zhao2023a} provided the first explicit complexity result for a Frank-Wolfe method for minimizing log-homogeneous self-concordant barriers, including our problem as a special case.
While NoLips, QEM, and Frank-Wolfe have clear complexity guarantees,
%they have an undesirable dependence on the sample size.
their time complexities scale at least linearly with the sample size, which is undesirable when the sample size is large.

%As far as we know, due to absence of Lipschitz continuity and smoothness in the optimization problems (\ref{eq:quantum_problem},~\ref{eq:classical_problem}), there are only two stochastic first-order methods that solve the problems with clear complexity guarantees, namely stochastic Q-Soft-Bayes \citep{Lin2022a} and stochastic Q-LB-OMD \citep{Tsai2022a}.
%The mini-batch stochastic Q-Soft-Bayes has been studied by \citet{Lin2021b}, but they did not improve the time complexity.

Other batch methods lack explicit complexity guarantees and, as a result, are not comparable to our algorithm.
For instance, the convergence rates of proximal gradient methods \citep{Tran-Dinh2015a} and several variants of the Frank-Wolfe method \citep{Dvurechensky2023a,Carderera2021a,Liu2022a} involve unknown parameters.
Diluted iterative MLE (iMLE, \citet{Rehacek2007a,Goncalves2014a}) and entropic mirror descent (EMD) with Armijo line search \citep{Li2019a} are only guaranteed to converge asymptotically.
Ordered-subset EM \citep{Hudson1994a} for PET and iMLE \citep{Lvovsky2004a} for ML quantum state tomography are commonly used heuristics but do not converge in general \citep{Hudson1994a,Rehacek2007a}.

\ifdefined\arxiv
\section{Applications}
\else
\section{APPLICATIONS}
\fi
%This section introduces applications metioned in Section~\ref{sec:introduction}.
%The Poisson inverse problem in Section~\ref{sec:pip} and maximum-likelihood quantum state estimation in Section~\ref{sec:mlqst} are considered in the numerical experiments.

\subsection{Kelly's Criterion} \label{sec:kelly}
Denote by $\Delta_d$ the probability simplex in $\R^d$. 
Consider long-term investment in a market with $d$ investment alternatives.
%where $a_t(i)\geq 0$ is the ratio of the closing price to the opening price of the $i$-th investment alternative on the $t$-th day.
%We assume that the vector $a_t = (a_t(i))_i$ follows some distribution $P_t$.
Let $\{ a_t \}$ be a stochastic process taking values in $[0,\infty)^d$.
On day $t$, the investor first selects a portfolio $x_t\in\Delta_d$ that indicates the distribution of their assets among the investment alternatives. 
%representing the distribution of the investor's assets over $d$ investment alternatives.
%, where $\Delta_d\coloneqq\{x\in[0,\infty)^d \mid \norm{x}_1 = 1\}$ denotes the probability simplex and $x_t$ represents the distribution of the investor's assets over $d$ investment alternatives.
%The investor then observes $a_t$, which lists the price relatives of the investment alternatives for the $t$-th day. 
% and the assets grow at a rate of $\braket{a_t,x_t}$.
%The goal of the investor is to maximize the growth rate in the long run.
Then, the investor observes $a_t$ that provides the price relatives of the investment alternatives for that day. 
The investor's goal is to maximize the wealth growth rate.

Kelly's criterion suggests choosing $x_{t+1}$ by maximizing the expected logarithmic loss
%growth rate 
conditional on the past \citep{Algoet1988a}, i.e.,
\[
	x_{t+1} \in \argmin_{x\in\Delta_d} \E_{a_{t+1}}[ -\log\braket{a_{t+1}, x} | a_1,\ldots, a_t],
\]
which requires solving the classical setup \eqref{eq:classical_problem}.
%where $P_t$ represents the probability distribution of the price relatives on the $t$-th day.
%Kelly's criterion has been proved to be asymptotically optimal for long-term investment \citep{Algoet1988a}.
%Note that computing $x_{t+1}$ requires solving the classical setup \eqref{eq:classical_problem}.

\subsection{Poisson Inverse Problem} \label{sec:pip}
In a Poisson inverse problem, our goal is to recover an unknown signal $\smash{\lambda^\natural\in [0,\infty)^d}$ based on $n$ independent measurement outcomes $\{y_i\}$.
Each outcome $y_i$ follows a Poisson distribution with mean $\braket{b_i,\lambda^\natural}$, where $b_i\in[0,\infty)^d$ is known and depends on the measurement setup.
In positron emission tomography, $\lambda^\natural(i)$ represents the emitter density of the $i$-th region, and $y_i$ represents the number of photons detected by the $i$-th sensor.

The ML estimate is given by \citep{Shepp1982a}
\begin{equation} \label{eq:pip}
\hat{\lambda}\in\argmin_{\lambda\in[0,\infty)^d } \sum_{i=1}^n (\braket{b_i, \lambda} - y_i \log\braket{b_i,\lambda}).
\end{equation}
\citet{Vardi1998a} and \citet{Ben-Tal2001a} showed that by setting
\[
	Y = \sum_{i=1}^n y_i,\quad \hat{\lambda}(i) = \frac{Y \hat{x}(i)}{\sum_{j=1}^n a_j(i)}, \quad \forall i\in\brackets{n},
\]
and
\[
	a_i(j) = \frac{Yb_i(j)}{\sum_{k=1}^n b_k(j)}, \quad \forall i \in \brackets{n},\ j \in \brackets{d},
\]
the ML estimate can be reformulated as
\[
	\hat{x} \in \argmin_{x\in\Delta_d} \sum_{i=1}^n -\frac{y_i}{Y} \log \braket{a_i, x},
\]
which is equivalent to the classical setup \eqref{eq:classical_problem} with $P'(a = a_i) = y_i/Y$ for all $i\in\brackets{n}$.

%\subsection{Postulates of Quantum Mechanics}
%We provide necessary backgrounds for the mathematical formulation of quantum mechanics before introducing quantum state estimation.
%
%\begin{itemize}
%\item
%A quantum state is described by a \textit{density matrix} $\rho\in\D_d$, where
%\[
%\D_d\coloneqq\{\rho\in\C^{d\times d}\mid \rho=\rho^\ast,\ \rho\geq 0,\ \tr\rho = 1\}
%\]
%is the set of density matrices.
%For a state consisting of $q$ qubits, $d=2^q$.
%The set $\D_d$ can be regarded as a quantum generalization of the probability simplex $\Delta_d$, in the sense that the vector of eigenvalues of $\rho\in\D_d$ lies in $\Delta_d$.
%\item
%To observe a quantum state, it must be measured using a \textit{positive operator-valued measure} (POVM).
%A POVM is described by a set $\mathcal{M}$ of Hermitian PSD matrices $A_1,\ldots,A_K\in\H_+^d$ that sum to the identity matrix, i.e., $A_{1:K} = I$.
%\item
%When we measure a quantum state $\rho$ using a POVM $\mathcal{M}=\{ A_k \mid k\in\brackets{K} \}$, the state collapses, losing information about the original state, and we observe a random outcome $\xi$ taking values in $\brackets{K}$, where
%\[
%	\P(\xi = k) = \tr(A_k\rho),\quad \forall k \in \brackets{K}.
%\]
%\end{itemize}

\subsection{ML Quantum State Tomography} \label{sec:mlqst}
A quantum state is described by a density matrix $\rho\in\D_d$, which is a $d\times d$ Hermitian PSD matrix of unit trace.
For a state consisting of $q$ qubits, $d$ equals $2^q$.
Denote by $\D_d$ the set of density matrices.
%\[
%\D_d\coloneqq\{\rho\in\C^{d\times d}\mid \rho=\rho^\ast,\ \rho\geq 0,\ \tr\rho = 1\}
%\]
%is the set of density matrices.
The set $\D_d$ can be regarded as a quantum generalization of the probability simplex $\Delta_d$, in the sense that the vector of eigenvalues of any $\rho\in\D_d$ lies in $\Delta_d$.

%Since a quantum state is not directly observable and loses information upon measurement,
%Identifying a quantum state requires a large number of identically prepared copies of it, followed by a series of measurements.
%Within this procedure, determining the required number of copies and selecting the appropriate measurements are statistical problems.
%Meanwhile, efficiently estimating the quantum state from the measurement outcomes is a computational problem, which is the primary focus of this work.

Given $n$ measurement outcomes from an unknown quantum state $\rho^\natural$, ML estimation is a standard and widely used approach to estimate $\rho^\natural$ \citep{Hradil1997a,Haffner2005a,Palmieri2020a,Brown2023a}.
%, produced independently by $n$ possibly different POVMs $\{\mathcal{M}_i\}$.
%A standard approach is ML estimation.
The ML estimate is given by
\[
	\hat{\rho}_{\text{ML}} \in \argmin_{\rho\in\D_d} \frac{1}{n} \sum_{i=1}^n -\log\tr( A_i \rho ),
\]
%where $A_i\in\mathcal{M}_i$ is the POVM element corresponding to the $i$-th measurement outcome.
for some known $A_i\in\H_+^d$ related to the $i$-th measurement outcome.
Note that computing $\hat{\rho}_{\text{ML}}$ requires solving the quantum setup \eqref{eq:quantum_problem} in the finite-sum setting with sample size $n$.
%It was recently shown that $n=\Omega(d^3/\delta^2)$ samples are required to ensure the closeness of $\hat{\rho}_{\text{ML}}$ and $\rho^\natural$ \citep{Chen2023a}.

\subsection{PSD Matrix Permanents} \label{sec:permanent}
The permanent of a matrix $A\in\C^{d\times d}$ is defined as
\[
	\per A \coloneqq \sum_{\pi\in S_d} \sum_{i=1}^d A_{i,\pi(i)},
\]
where $S_d$ is the set of all permutations of $\brackets{d}$.
Let $\{ v_i \}_{i=1}^d$ be the eigenvectors of $A$.
\citet{Yuan2022a} proposed the following approximation of $\per A$ when $A\in\H_+^d$:
\[
	\rel A \coloneqq \max_{\rho\in\D_d} \prod_{i=1}^d \tr((dv_i v_i^\ast)\rho).
\]
The approximation is equivalent to the quantum setup \eqref{eq:quantum_problem} in the finite-sum setting with $A_i = dv_iv_i^\ast$.
As noted by \citet{Meiburg2022a}, $\rel A$ achieves the currently tightest approximation ratio of $4.85^d$.

% !TEX root = ../paper.tex
\ifdefined\arxiv
\section{Characterizations of Logarithmic Loss}
\else
\section{CHARACTERIZATIONS OF LOGARITHMIC LOSS}
\fi
\label{sec:properties}
This section aims to address the aforementioned lack of Lipschitz continuity and smoothness in the logarithmic loss.
%As mentioned earlier, the loss function in the optimization problem \eqref{eq:quantum_problem} violates the standard Lipschitz continuity and smoothness conditions in first-order convex optimization.
%This section aims to address these issues.
We first set up a few notations.
For $\rho\in\H_{++}^d$, let 
\begin{equation} \label{eq:log_barrier}
	h(\rho)\coloneqq-\log\det\rho
\end{equation}
be the \textit{logarithmic barrier}.
Let $\smash{\norm{\cdot}_\rho \coloneqq (D^2 h(\rho)[\cdot,\cdot])^{1/2}}$ be the \textit{local norm} associated with $h$ at $\rho\in\dom h$.
The following lemma gives explicit formulae for the local norm and its dual norm.
The proof is deferred to Appendix~\ref{app:local_norm}.

\begin{lemma}
For $\rho\in\H^d_{++}$ and $X\in\H^d$, the local norm and its dual norm associated with $h$ are given by
\begin{equation} \label{eq:local_norm}
\begin{split}
	&\norm{X}_\rho = \sqrt{\tr((\rho^{-1/2}X\rho^{-1/2})^2)} = \sqrt{\tr((\rho^{-1}X)^2)}, \\
	&\norm{X}_{\rho,\ast} = \sqrt{\tr((\rho^{1/2}X\rho^{1/2})^2)} = \sqrt{\tr((\rho X)^2)}.
\end{split}
\end{equation}
\end{lemma}

\subsection{``Lipschitz Continuity''}
A continuously differentiable function $f:\H^d\to\R$ is said to be $G$-Lipschitz with respect to a norm $\norm{\cdot}$ if its gradients are bounded by $G$ in the dual norm, i.e., $\norm{\nabla f(\rho)}_\ast \leq G$.

Although the loss function is not Lipschitz, Lemma~\ref{lem:lipschitz} below shows that $\nabla f$ is bounded in the dual local norm associated with $h$.
This Lipschitz-type property enables us to control the distance between iterates and exploit local properties of the loss function, in particular, the local smoothness property of self-concordant functions (Theorem~\ref{thm:self_concordant_smoothness}).

Lemma~\ref{lem:lipschitz} is a simple quantum generalization of Lemma~4.3 of \citet{Tsai2023b}.
Its proof is deferred to Appendix~\ref{app:lipschitz}.

\begin{lemma}\label{lem:lipschitz}
Let $f$ be defined in the quantum setup \eqref{eq:quantum_problem}.
Then, $\norm{\nabla f(\rho)}_{\rho,\ast}\leq 1$ for all $\rho\in\H_{++}^d$.
\end{lemma}

\subsection{``Smoothness''}
A continuously differentiable function $f:\H^d\to\R$ is said to be $L$-smooth with respect to a norm $\norm{\cdot}$ if its gradient is $L$-Lipschitz with respect to $\norm{\cdot}$, i.e.,
\[
	\norm{\nabla f(\rho) - \nabla f(\rho')}_\ast \leq L\norm{\rho-\rho'},\quad \forall \rho,\rho'\in\H^d.
\]
Lemma~\ref{lem:standard_self_bounding}, known as the self-bounding property, is a consequence of smoothness \citep{Srebro2010a}. 
A proof of Lemma~\ref{lem:standard_self_bounding} can be found in Lemma~4.23 of \citet{Orabona2023a}.
% It has been explicitly stated in the work of \citet{Srebro2010a} for the special case when the norm is the $\ell_2$-norm.

\begin{lemma} \label{lem:standard_self_bounding}
Let $f:\R^d\to\R$ be $L$-smooth with respect to $\norm{\cdot}$ with $\dom f=\R^d$.
Then, for any $x\in\R^d$, it holds that
\[
	\norm{ \nabla f(x) }_\ast^2 \leq 2L \left( f(x) - \inf_{x'\in\R^d} f(x') \right).
\]
\end{lemma}

Although the loss function is not smooth, Lemma~\ref{lem:self_bounding} below establishes a self-bounding-type property of the loss function.
As discussed in Section \ref{sec:introduction}, the lemma generalizes Lemma~4.7 of \citet{Tsai2023b} to the quantum setup, and greatly simplifies the proof therein.
% The generalization is challenging due to the noncommutativity.
The proof is deferred to Appendix~\ref{app:self_bounding}.

For any $\rho\in\ri\D_d$ and $X\in\H^d$, define
\begin{equation} \label{eq:alpha}
	\alpha_\rho(X) \coloneqq - \frac{\tr(\rho X\rho)}{\tr(\rho^2)} \in \argmin_{\alpha\in\R} \norm{X + \alpha I}_{\rho,\ast}^2.
\end{equation}

\begin{lemma} \label{lem:self_bounding}
Let $f$ be defined in the quantum setup \eqref{eq:quantum_problem}.
Then, for any $\rho\in\ri\D_d$, it holds that
\begin{equation*}
\norm{ \nabla f(\rho) + \alpha_\rho(\nabla f(\rho))I }_{\rho,\ast}^2 \leq 4\left( f(\rho) - \min_{\rho'\in\D_d} f(\rho') \right).
\end{equation*}
\end{lemma}

% !TEX root = ../paper.tex
\ifdefined\arxiv
\section{Algorithms and Convergence Guarantees}
\else
\section{ALGORITHMS AND CONVERGENCE GUARANTEES}
\fi
This section presents LB-SDA and its theoretical guarantee.
We focus on the quantum setup \eqref{eq:quantum_problem} since it includes the classical setup \eqref{eq:classical_problem} as a special case.

\subsection{Algorithm}

\begin{algorithm}[t!]
\caption{Stochastic Dual Averaging with the Logarithmic Barrier (LB-SDA) for the quantum setup}
\label{alg:sda}
\hspace*{\algorithmicindent} \textbf{Input: } A stochastic first-order oracle $\O$.
\begin{algorithmic}[1]
\STATE $h ( \rho ) \coloneqq -\log\det\rho$. 
\STATE $\rho_1 = I/d \in \argmin_{\rho\in\D_d} h(\rho)$.
\FORALL{$t \in \mathbb{N}$}
	\STATE Output $\bar{\rho}_t\coloneqq(1/t)\rho_{1:t}$.
	\STATE $g_t = \O(\rho_t)$.
	\STATE Compute a learning rate $\eta_t>0$.
	\STATE $\smash{ \rho_{t+1}\in\argmin_{\rho\in\D_d} \eta_t\tr(g_{1:t}\rho) + h(\rho) }$.
\ENDFOR
\end{algorithmic}
\end{algorithm}

LB-SDA is presented in Algorithm~\ref{alg:sda}, where $h$ is the logarithmic barrier \eqref{eq:log_barrier} and $\norm{\cdot}_\rho$ and $\norm{\cdot}_{\rho,\ast}$ are the local and dual local norms associated with $h$ at $\rho$ \eqref{eq:local_norm}, respectively.
A \textit{stochastic first-order oracle} is a randomized function $\O$ that outputs an unbiased estimate $\O(\rho)\in\H^d$ of the gradient $\nabla f(\rho)$ given an input $\rho\in\D_d$.

We will make the following assumptions on the stochastic first-order oracle.
It is notable that the boundedness is defined in terms of the dual local norm, which deviates from existing literature.

\begin{assumption} \label{ass:stochastic_gradients}
Conditional on the past, the stochastic gradients $\{g_t\}$ in Algorithm~\ref{alg:sda} are unbiased and bounded, and their variances are also bounded, i.e., for all $t\in\N$,
\begin{itemize}
\item $\E[g_t| \mathcal{H}_t] = \nabla f(\rho_t)$,
\item ${ \E\left[ \norm{g_t}_{\rho_t,\ast}^2 \big|\mathcal{H}_t \right] \leq G^2 }$,
\item ${ \E\left[\norm{g_t - \nabla f(\rho_t)}_{\rho_t,\ast}^2 \big| \mathcal{H}_t \right] \leq\sigma^2 }$,
\end{itemize}
where $\mathcal{H}_t = \{ g_1,\ldots, g_{t-1}, \rho_1,\ldots, \rho_t \}$ is the past information before obtaining $g_t$.
\end{assumption}

The unbiasedness and bounded variance assumptions are standard in the literature.
Regarding the bounded gradient assumption, by the triangle inequality and Lemma~\ref{lem:lipschitz},
\[
\begin{split}
	\norm{g_t}_{\rho_t,\ast}^2 &\leq (\norm{g_t - \nabla f(\rho_t)}_{\rho_t,\ast} + \norm{\nabla f(\rho_t)}_{\rho_t,\ast})^2 \\
	&\leq \norm{g_t - \nabla f(\rho_t)}_{\rho_t,\ast}^2 + 2\norm{g_t - \nabla f(\rho_t)}_{\rho_t,\ast} + 1.
\end{split}
\]
Taking expectations on both sides and using the inequality $\E X \leq \sqrt{\E[X^2]}$, we can verify that the bounded gradient assumption always holds with $G=1+\sigma$.
Nevertheless, since $G$ can be smaller than $1+\sigma$, we include the assumption for a tighter result.

An important example of the oracle is
\begin{equation} \label{eq:oracle_b}
	\O_B(\rho) \coloneqq \frac{1}{B}\sum_{b=1}^B \nabla \ell_b(\rho)
\end{equation}
where $B\in\N$, $\ell_b(\rho)\coloneqq-\log\tr(A_b\rho)$, and $A_1,\ldots,A_B$ are independently drawn from $P$.
The resulting algorithm is called \textit{$B$-sample LB-SDA}.
The following lemma justifies the use of $\O_B$, whose proof is deferred to Appendix~\ref{app:oracle}.

\begin{lemma} \label{lem:oracle_b}
The oracle $\O_B$ \eqref{eq:oracle_b} satisfies Assumption~\ref{ass:stochastic_gradients} with $G=1$ and $\sigma^2=4/B$.
\end{lemma}

\subsection{Convergence Guarantee}

The non-asymptotic convergence guarantee of Algorithm~\ref{alg:sda} is presented in Theorem~\ref{thm:sda} below.
The analysis follows the online-to-batch approach, where we use the following regret bound of \citet{Tsai2023b} in Appendix~\ref{app:olo}:
\[
	R_t \leq \tildeO\left( \sqrt{d \sum_{\tau=1}^t \norm{g_\tau + \alpha_{\rho_\tau}(g_\tau) I}_{\rho_\tau,\ast}^2  } \right)
	\leq \tildeO( \sqrt{ dt } ).
\]
%where we first analyze an algorithm in online convex optimization, and then convert it to a stochastic algorithm.
Note that applying the online-to-batch conversion on the right upper bound can only yield a convergence rate of $\tildeO(\sqrt{d/t})$, independent of the variance $\sigma^2$.
Since the effect of batch size is unclear without the variance term, this direct approach fails to improve the time complexity guarantee.

Deriving an error bound involving a variance term typically requires smoothness of the loss function in the literature, and this is where the self-bounding-type property (Lemma~\ref{lem:self_bounding}) comes into play.
It bounds the square of the dual local norm by
\[
	\E\norm{g_\tau + \alpha_{\rho_\tau}(g_\tau) I}_{\rho_\tau,\ast}^2
	\leq 4\E\left[ f(\rho_\tau) - \min_{\rho\in\D_d} f(\rho) \right] + \sigma^2,
\]
which results in a ``self-bounding'' inequality of $\E R_t$:
\[
	\E R_t \leq \tildeO\left( \sqrt{d \E R_t + \sigma^2dt} \right)
\]
Our analysis can be seen as a local-norm extension of that of the anytime online-to-batch conversion \citep{Cutkosky2019a}.
The proof is deferred to Appendix~\ref{app:sda_theorem}.

%In Theorem~\ref{thm:sda}, the second term involving the variance $\sigma^2$ is crucial to characterize the effect of batch size, and hence critical to improve the time complexity.

%is not able to characterize the effect of batch size, and does not lead to the time complexity improvement.
%
%Deriving this term typically requires smoothness of the loss function.
%This is where the self-bounding-type property (Lemma~\ref{lem:self_bounding}) comes into play.

%The online part generalizes Theorem~3.2 of \citet{Tsai2023b} to the quantum setup.
%We use the standard online-to-batch conversion \citep{Cesa-Bianchi2004a} because its empirical performance seems to be better.

\begin{theorem} \label{thm:sda}
Consider the quantum setup \eqref{eq:quantum_problem}.
Under Assumption~\ref{ass:stochastic_gradients}, let $\{\bar{\rho}_t\}$ be the iterates generated by Algorithm~\ref{alg:sda} with
\[
	\eta_t = \frac{\sqrt{d}}{\sqrt{\sum_{\tau=1}^t \norm{g_\tau + \alpha_{\rho_\tau}(g_\tau) I }_{\rho_\tau,\ast}^2 + 4dG^2 + G^2}}.
\]
Then, for all $t\in\N$, it holds that
\[
\setlength{\jot}{10pt}
\begin{split}
	&\E\left[ f(\bar{\rho}_t) - \min_{\rho\in\D_d} f(\rho) \right] \\
	&\qquad \leq
	\frac{4dC_t^3 + 2C_t\sqrt{\sigma^2 dt + 4d^2G^2 + dG^2} + 1}{t} \\ 
	&\qquad = O\left( \frac{dG(\log t)^3}{t} + \frac{\sigma\sqrt{d}\log t}{\sqrt{t}} \right),
\end{split}
\]
where $C_t \coloneqq \log t + 3$ and the expectation is taken with respect to $\{g_t\}$.
\end{theorem}

%The convergence rate in Theorem~\ref{thm:sda} matches standard results of first-order methods in smooth stochastic optimization \citep{Dekel2012a}.

%Corollary~\ref{cor:sda} states the convergence rate for $B$-sample LB-SDA, which is a direct consequence of Theorem~\ref{thm:sda} and Lemma~\ref{lem:oracle_b}.
%The proof is hence omitted.
Plugging in the estimates in Lemma~\ref{lem:oracle_b}, we obtain the following result for the mini-batch case. 

\begin{corollary} \label{cor:sda}
Consider the quantum setup \eqref{eq:quantum_problem}.
Let $\{\bar{\rho}_t\}$ be the iterates of B-sample LB-SDA with learning rates
\[
	\eta_t = \frac{\sqrt{d}}{\sqrt{\sum_{\tau=1}^t \norm{g_\tau + \alpha_{\rho_\tau}(g_\tau) I }_{\rho_\tau,\ast}^2 + 4d + 1}}.
\]
Then, for all $t\in\N$, it holds that
\[
%\setlength{\jot}{10pt}
%\begin{split}
	\E\left[ f(\bar{\rho}_t) - \min_{\rho\in\D_d} f(\rho) \right]
%	&\qquad \leq
%	\frac{4dC_t^2 + 2C_t\sqrt{4dt/B + 4d^2 + d} + 1}{t} \\
	= O\left( \frac{d(\log t)^3}{t} + \frac{\sqrt{d}\log t}{\sqrt{Bt}} \right).
%\end{split}
\]
where the expectation is taken with respect to $\{g_t\}$.
%\begin{enumerate}
%\item
%$B$-sample LB-SDA satisfies
%\[
%\begin{split}
%	&\E\left[ f(\bar{\rho}_t) - \min_{\sigma\in\D_d} f(\sigma) \right] \\
%	&\qquad \leq
%	\frac{4d(2+\log t)^2 + 2(2+\log t)\sqrt{(4/B) dt + 4d^2 + d}}{t} \\
%	&\qquad = O\left( \frac{d(\log t)^2}{t} + \frac{\sqrt{d}\log t}{\sqrt{Bt}} \right).
%\end{split}
%\]
%
%\item
%LB-DA satisfies
%\[
%\begin{split}
%	&f(\bar{\rho}_t) - \min_{\sigma\in\D_d} f(\sigma) \\
%	&\qquad \leq \frac{4d(2+\log t)^2 + 2(2+\log t)\sqrt{4d^2 + d}}{t} \\
%	&\qquad = O\left( \frac{d(\log t)^2}{t} \right).
%\end{split}
%\]
%\end{enumerate}
\end{corollary}

\begin{remark}
Proving high-probability guarantees for $B$-sample LB-SDA is challenging since the logarithmic loss violates the boundedness assumption required by standard analysis \citep{Orabona2023a}.
We left this extension as a future research direction.
\end{remark}

%\begin{remark}
%In the finite-sum setting, another measure of the ``noise'' is the finite optimal loss function difference \citet{} defined as
%\[
%	\Delta^2 \coloneqq \min_{\rho\in\D_d} f(\rho) - \frac{1}{n}\sum_{i=1}^n \min_{\rho\in\D_d} f_i(\rho),
%\]
%where $f_i(\rho)\coloneqq-\log\tr(A_i\rho)$.
%\end{remark}

\subsection{Time Complexity Analysis} \label{sec:time_complexity}

This section discusses the time complexity of $B$-sample LB-SDA.
Comparisons of time complexities of existing first-order methods have been presented in Table~\ref{tab:time_complexity_comparison} and discussed in Section~\ref{sec:introduction}.

First, note that the 6th line in Algorithm~\ref{alg:sda} cannot be solved exactly.
Nevertheless, after an eigendecomposition, which takes $\tildeO(d^\omega)$ time \citep{Demmel2007a}, the 6th line reduces to an one-dimensional convex optimization problem, which can be efficiently solved by Newton's method on the real line in $\tildeO(d)$ time (see, e.g., Appendix~A.2 of \citet{Nesterov2018a}).
% and Appendix~B of \citet{Kotlowski2019a}).
As a result, the time complexity of the 6th line is $\tildeO(d^\omega)$.
Second, the time complexity of the 5th line, which requires implementing the oracle $\O_B$, is $O(Bd^2)$.
Lastly, since the 5th and the 6th lines are the most time-consuming parts, the per-iteration time complexity of $B$-sample LB-SDA is $O(Bd^2+d^\omega)$.

By Corollary~\ref{cor:sda}, the iteration complexity of $B$-sample LB-SDA to obtain an $\varepsilon$-optimal solution is $\tildeO(d/(B\varepsilon^2))$.
Combining with the per-iteration time complexity, the overall time complexity is $\tildeO(d^3/\varepsilon^2 + d^{\omega+1}/(B\varepsilon^2))$.
In particular, the overall time complexity is $\tildeO(d^3/\varepsilon^2)$ when $B = \Omega(d^{\omega-2})$.
Since $\omega$ is $3$ in practical implementation \citep{Dongarra1990a}, we will often choose $B=d$.

%As presented in Table~\ref{tab:time_complexity_comparison}, $d$-sample LB-SDA achieves the currently best time complexity among existing stochastic methods.
%%improving upon previous results \citep{Lin2022a,Tsai2022a} by a factor of $d^{\omega-2}$.
%On the other hand, the time complexity of $d$-sample LB-SDA does not depend on the sample size $n$, so it can be faster than batch methods when $n\gg d$.
%For instance, in ML quantum state tomography, $n=\Omega(d^3)$ is necessary to obtain an accurate estimate \citep{Chen2023a}x.

Since the eigendecomposition is no longer needed in the classical setup, the per-iteration time complexity of $B$-sample LB-SDA is reduced to $\tildeO(Bd)$.
Because the iteration complexity of obtaining an $\varepsilon$-optimal solution is $\tildeO(d/(B\varepsilon^2))$, the overall time complexity is $\tildeO(d^2/\varepsilon^2)$ for any $B\in\N$ in the classical setup.
%This aligns with the state of the art of stochastic first-order methods \citep{Lin2022a,Tsai2022a}.
% !TEX root = ../paper.tex
\ifdefined\arxiv
\section{Numerical Results}
\else
\section{NUMERICAL RESULTS} 
\fi
\label{sec:numerical_results}

We have shown that LB-SDA achieves the currently best time complexity guarantees in the previous section.
In this section, we show that LB-SDA also performs well empirically.
We consider solving the Poisson inverse problem and computing the ML estimate for quantum state tomography.
All results in this section are presented in terms of the elapsed time.
Results in terms of the number of iterations can be found in Appendix~\ref{app:add_numerical_results}.

Both experiments were conducted on a machine with an Intel Xeon Gold 5218 CPU of 2.30GHz and 131,621,512kB memory.
The elapsed time records the actual running time of the method on the machine.
All methods are implemented in the Julia programming language \citep{Bezanson2017a} with the Intel Math Kernel Library, and the number of threads in BLAS is set to $8$.
It is important to note that the empirical speed is highly dependent on the specific implementations.
The source code of the experiments is available at \url{https://github.com/chungentsai/pip} and \url{https://github.com/chungentsai/mlqst} for the Poisson inverse problem and ML quantum state tomography, respectively.

The approximate optimization error at an iterate is defined as the difference between its function value and the smallest one obtained in the experiments.

\subsection{Poisson Inverse Problem}
\label{sec:numerical_pip}

Consider the Poisson inverse problem in Section~\ref{sec:pip} with a synthetic dataset, where $d$ equals $256$ and $n$ equals $1,000,000$.
The unknown signal $\lambda^\natural$ is $1,000$ times the gray intensities of the Shepp-Logan phantom image \citep{Shepp1974a} of size $16\times 16$.
The signal is presented in Appendix~\ref{app:add_numerical_results}.
The vectors $\{ b_i \}$ are generated following the scheme of \citet{Raginsky2010a}.
Each entry of $b_i$ is assigned to either $0$ or $1/n$ with equal probability.

We consider all algorithms in Table~\ref{tab:time_complexity_comparison} that have explicit complexity guarantees.
EM \citep{Shepp1982a}, SSB \citep{Li2020a}, and SLBOMD \citep{Tsai2022a} are the classical counterparts of QEM, SQSB, and SQLBOMD, respectively.
Additionally, we include SPDHG \citep{Chambolle2018a} and EMD with Armijo line search \citep{Li2018a} for comparison.
The former is well-known in practice and the latter is known to converge fast empirically.
However, they are only guaranteed to converge asymptotically.
Their parameters are set according to the cited works.
We do not include batch PDHG as it is slow in practice.

We solve the Poisson inverse problem based on the equivalence between it and the classical setup \eqref{eq:classical_problem} in Section~\ref{sec:pip}.
%In particular, let $\hat{x}\in\Delta_d$ be an iterate of the algorithm for solving the classical setup \eqref{eq:classical_problem}.
%Then, the estimated signal $\hat{\lambda}$ is computed as
%\[
%	\hat{\lambda}(i) = \frac{Y \hat{x}(i)}{\sum_{j=1}^n a_j(i)}, \quad \forall i\in\brackets{d}.
%\]
\begin{figure}[t]
\caption{Performances of all algorithms in Table~\ref{tab:time_complexity_comparison}, SPDHG, and EMD with line search for solving the Poisson inverse problem.}
\label{fig:pip}
\centering

\ifdefined\arxiv
	\begin{subfigure}{0.484\columnwidth}
	\centering
	\caption{Normalized estimation error versus the elapsed time.}
	\includegraphics[width=\columnwidth]{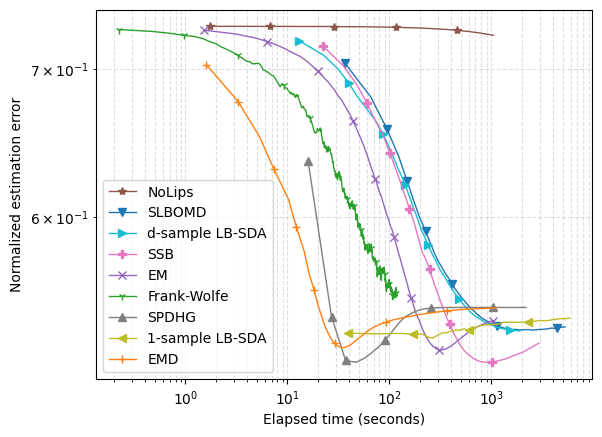}
	\end{subfigure}%
	\hfill
	\begin{subfigure}{0.48\columnwidth}
	\centering
	\caption{Approximate optimization error versus the elapsed time.}
	\includegraphics[width=\columnwidth]{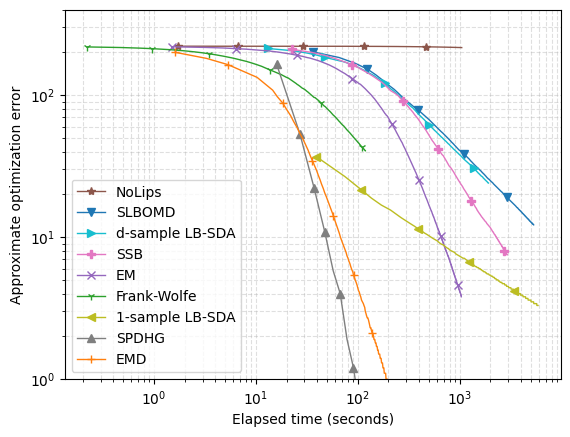}
	\end{subfigure}%
\else
	\begin{subfigure}{\columnwidth}
	\centering
	\caption{Normalized estimation error versus the elapsed time.}
	\includegraphics[width=\columnwidth]{./figures/pip/phantom16/time-distance.png}
	\end{subfigure}%
	\vfill
	\begin{subfigure}{\columnwidth}
	\centering
	\caption{Approximate optimization error versus the elapsed time.}
	\includegraphics[width=\columnwidth]{./figures/pip/phantom16/time-error.png}
	\end{subfigure}%
\fi

\end{figure}

Figure~\ref{fig:pip} presents the numerical results.
For an iterate $\hat{\lambda}$, the normalized estimation error is defined as $\norm{\hat{\lambda} - \lambda^\natural}_2 / \norm{\lambda^\natural}_2$.
Since the goal of the Poisson inverse problem is to recover the unknown signal $\lambda^\natural$, rather than minimizing the loss function, results presented in terms of the normalized estimation error is more important than results presented in terms of the optimization error.

Observe that $1$-sample LB-SDA outperforms all methods with explicit complexity guarantees in terms of the normalized estimation error.
Although it is slower than EMD with line search and SPDHG, the latter two methods are only guaranteed to converge asymptotically, whereas LB-SDA has an explicit non-asymptotic complexity guarantee.

LB-SDA converges faster than SLBOMD and SSB in terms of the optimization error, although they have the same theoretical time complexity of $\tildeO(d^2/\varepsilon^2)$.
%In terms of the optimization error, even though LB-SDA, SLBOMD, and SSB have the same theoretical time complexity of $\tildeO(d^2/\varepsilon^2)$, LB-SDA converges the fastest.
This can be explained by the use of time-varying learning rates in LB-SDA, in contrast to the fixed learning rates used by the other two methods.
The time-varying learning rates are large at the beginning, which leads to a fast convergence in practice.

\subsection{ML Quantum State Tomography}
\label{sec:numerical_mlqst}

Consider the problem of ML quantum state tomography in Section~\ref{sec:mlqst}.
We construct a synthetic dataset, following the setup of \citet{Haffner2005a}.
The number of qubits $q$ is $6$, the dimension $d$ is $2^6=64$, and the sample size $n$ is $409,600$.
The unknown quantum state is the $W$ state, which corresponds to a rank-$1$ density matrix.
The Hermitian matrices $\{ A_i \}$ are generated following the procedure of \citet{Lin2021a}, where each $A_i$ is of rank $d/2$.
%We use random Pauli observables as our POVMs, where each element is of rank $d/2$.

We compare all algorithms in Table~\ref{tab:time_complexity_comparison}, along with iMLE \citep{Lvovsky2004a}, diluted iMLE \citep{Goncalves2014a}, and EMD with Armijo line search \citep{Li2018a}.
Their parameters are set according to the cited works.
Although iMLE does not always converge \citep{Rehacek2007a}, we include it because it is often considered as a benchmark.
We do not include the accelerated projected gradient descent \citep{Shang2017a} as it is slower than iMLE in experiments \citep{Ahmed2021a}.

\begin{figure}[t!]
\caption{Performances of all algorithms in Table~\ref{tab:time_complexity_comparison}, iMLE, diluted iMLE, and EMD with line search for computing the ML estimate for quantum state tomography.}
\label{fig:mlqst}
\centering

\ifdefined\arxiv
	\begin{subfigure}{0.48\columnwidth}
	\centering
	\caption{Fidelity between the iterates and the $W$ state versus the elapsed time.}
	\includegraphics[width=\columnwidth]{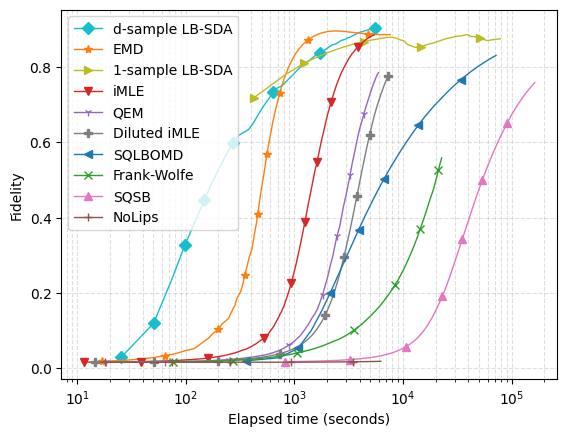}
\end{subfigure}%
	\hfill
\begin{subfigure}{0.48\columnwidth}
	\centering
	\caption{Approximate optimization error versus the elapsed time.}
	\includegraphics[width=\columnwidth]{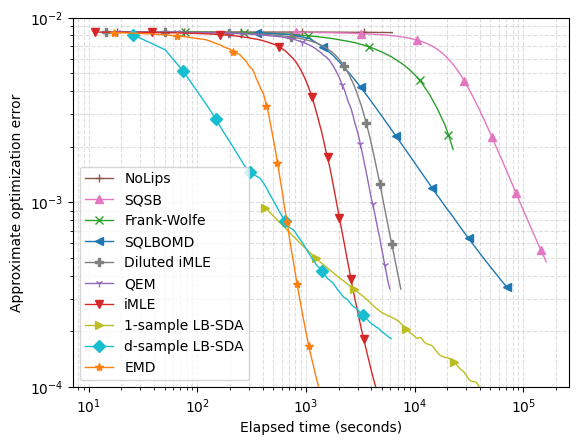}
	\end{subfigure}%
\else
	\begin{subfigure}{\columnwidth}
	\centering
	\caption{Fidelity between the iterates and the $W$ state versus the elapsed time.}
	\includegraphics[width=\columnwidth]{./figures/mlqst/6qubits/time-fidelity.png}
	\end{subfigure}%
	\vfill
	\begin{subfigure}{\columnwidth}
	\centering
	\caption{Approximate optimization error versus the elapsed time.}
	\includegraphics[width=\columnwidth]{./figures/mlqst/6qubits/time-error.png}
	\end{subfigure}%
\fi

\end{figure}

Figure~\ref{fig:mlqst} presents the numerical results.
The fidelity between two quantum states $\rho,\rho'\in\D_d$ is defined as $\smash{F(\rho,\rho')\coloneqq\left( \tr \sqrt{ \sqrt{\rho}\rho'\sqrt{\rho} } \right)^2}\in[0,1]$.
It is a standard measure of the closeness of two quantum states, with $F(\rho,\rho')=1$ if and only if $\smash{\rho=\rho'}$.
Similar to the Poisson inverse problem, as the goal of quantum state tomography is to recover the unknown quantum state, results presented in terms of the fidelity is more important than results presented in terms of the optimization error.

Observe that $d$-sample LB-SDA outperforms all methods in terms of the fidelity.
We conclude that $d$-sample LB-SDA achieves the currently best theoretical time complexity and the currently best empirical performance for computing the ML estimate for quantum state tomography.

Note that $d$-sample LB-SDA performs better than SQLBOMD and SQSB in terms of the optimization error.
%This confirms the time complexity analysis in Section~\ref{sec:time_complexity}.
It also outperforms QEM, EMD with line search, diluted iMLE, and iMLE when the optimization error is not smaller than $10^{-3}$.
Recall that the latter four algorithms possess theoretical drawbacks. The time complexity of QEM has a worse sample size dependence and a better optimization error dependence than that of $d$-sample LB-SDA; EMD with line search and diluted iMLE lack non-asymptotic complexity guarantees; and iMLE does not converge in general.

While it is theoretically known that stochastic methods outperform batch ones when the dimension and the sample size are sufficiently large \citep{Bottou2007a}, empirical results presented in the literature did not confirm this phenomenon.
In this work, we observed that $d$-sample LB-SDA outperforms all methods in terms of the fidelity.
This marks the first empirical evidence that stochastic methods can be more efficient than batch methods for computing the ML estimate for quantum state tomography.

% !TEX root = ../paper.tex
\ifdefined\arxiv
\section{Concluding Remarks}
\else
\section{CONCLUDING REMARKS}
\fi
We have proposed a stochastic first-order method named $B$-sample LB-SDA for solving the Poisson inverse problem, computing the ML estimate for quantum state tomography, and approximating PSD matrix permanents.
In particular, $d$-sample LB-SDA takes $\tildeO(d^3/\varepsilon^2)$ time to obtain an $\varepsilon$-optimal solution in the quantum setup, improving the time complexities of existing first-order methods.
The improvement is based on a new analysis for mini-batch methods, which relies on a novel self-bounding-type property of the logarithmic loss and a new local-norm based analysis of the online-to-batch conversion.
Lastly, we have shown that LB-SDA performs better empirically than all methods with explicit complexity guarantees.

Several research directions arise.
One direction is to design accelerated or variance-reduced methods for solving the optimization problem \eqref{eq:quantum_problem} based on the smoothness characterization.
Another direction is to generalize our argument to other non-smooth loss functions.
%Another direction is to generalize our argument to a class of non-smooth functions.
%Such generalization seems to be highly non-trivial at this moment.
%Finally, considering that the empirical performance of LB-SDA in terms of the fidelity is better than that in terms of the optimization error, it will be interesting if one can explain such phenomenon theoretically.

\acks{C.-E.~Tsai and Y.-H.~Li are supported by the Young Scholar Fellowship (Einstein Program) of the National Science and Technology Council of Taiwan under grant number NSTC 112-2636-E-002-003, by the 2030 Cross-Generation Young Scholars Program (Excellent Young Scholars) of the National Science and Technology Council of Taiwan under grant number NSTC 112-2628-E-002-019-MY3, by the research project ``Pioneering Research in Forefront Quantum Computing, Learning and Engineering'' of National Taiwan University under grant number NTU-CC-112L893406, and by the Academic Research-Career Development Project (Laurel Research Project) of National Taiwan University under grant number NTU- CDP-112L7786.

H.-C.~Cheng is supported by Grants No.~NSTC 112-2636-E-002-009, No.~NSTC 112-2119-M-007-006, No.~NSTC 112-2119-M-001-006, No.~NSTC 112-2124-M-002-003, No.~NTU-112V1904-4, No.~NTU-112L900702, and by the research project ``Pioneering Research in Forefront Quantum Computing, Learning and Engineering'' of National Taiwan University under Grant No.~NTC-CC-112L893405.}

\appendix

% !TEX root = ../paper.tex
\ifdefined\arxiv
\section{Preliminaries}
\else
\section{PRELIMINARIES} 
\fi
\label{app:preliminaries}
Throughout this section, let $(\V,\braket{\cdot,\cdot})$ be a finite-dimensional real Hilbert space, such as $\R^d$ with the standard inner product and $\H^d$ with the Hilbert-Schmidt inner product $\braket{U,V}\coloneqq\tr(U^\ast V)$.
Let $\mathcal{X}\subseteq\V$ be a convex set.

\subsection{Self-Concordance and Relative Smoothness} \label{app:self_concordance_relative_smoothness}
This section provides necessary background information on the notions of \textit{self-concordance} \citep{Nesterov1994a,Nesterov2018a} and \textit{relative smoothness} \citep{Bauschke2017a,Lu2018a}, which form the basis of the smoothness characterization in Section~\ref{sec:properties}.
We begin with self-concordance.

\begin{definition}[Self-concordance]
A closed convex function $\varphi:\V\to(-\infty,\infty]$ with an open domain $\dom \varphi$ is said to be $M$-self-concordant if it is three-times continuously differentiable on $\dom\varphi$ and
\[
	\abs{ D^3 \varphi(x)[u,u,u] } \leq 2M( D^2\varphi(x)[u,u] )^{3/2},\quad\forall x\in\dom\varphi,\ u\in\V.
\]
\end{definition}

\begin{theorem}[Theorem~5.1.5 of \citet{Nesterov2018a}] \label{thm:dikin_ellipsoid}
Let $\varphi$ be an $M$-self-concordant function.
Let $\norm{\cdot}_x\coloneqq(D^2\varphi(x)[\cdot,\cdot])^{1/2}$ be the local norm associated with $\varphi$ at $x$.
Define the Dikin ellipsoid $W(x)\coloneqq\{ y\in\V \mid \norm{y-x}_x < 1/M \}$.
Then, $W(x)\subseteq\dom \varphi$ for all $x\in\dom\varphi$.
\end{theorem}

Theorem~\ref{thm:self_concordant_smoothness} presents an important local smoothness-type property of self-concordant functions.
Define $\omega(t) \coloneqq t - \log(1+t)$ and its Fenchel conjugate $\omega_\ast(t) = -t-\log(1-t)$.
\begin{theorem}[Theorem~5.1.9 and Lemma~5.1.5 of \citet{Nesterov2018a}] \label{thm:self_concordant_smoothness}
Let $\varphi$ be an $M$-self-concordant function.
Let $\norm{\cdot}_x\coloneqq(D^2\varphi(x)[\cdot,\cdot])^{1/2}$ be the local norm associated with $\varphi$ at $x$.
Then, for $x,y\in\dom\varphi$ such that $\norm{y-x}_x < 1/M$, it holds that
\[
	\varphi(y) \leq \varphi(x) + \braket{ \nabla \varphi (x), y - x} + \frac{1}{M^2}\omega_\ast( M\norm{ y-x }_x ).
\]
Moreover, if $\norm{y-x}_x < 1/(2M)$, then
\[
	\varphi(y) \leq \varphi(x) + \braket{ \nabla \varphi (x), y - x} + \norm{ y-x }_x^2.
\]
\end{theorem}

\begin{lemma}[Proposition~5.4.5 of \citet{Nesterov1994a}] \label{lem:self_concordance}
The logarithmic barrier $h(\rho) = -\log\det\rho$ is $1$-self-concordant.
\end{lemma}

Now, we introduce the notion of relative smoothness.
\begin{definition}[Relative smoothness] \label{def:relative_smoothness}
Let $f, h:\V\to(-\infty,\infty]$.
The function $f$ is said to be $L$-smooth relative to $h$ on $\mathcal{X}$ for some $L> 0$ if $Lh-f$ is convex on $\mathcal{X}$.
\end{definition}

\begin{lemma}[Proposition~7 of \citet{Tsai2023a}] \label{lem:relative_smoothness}
Let $f(\rho)\coloneqq\E[-\log\tr(A\rho)]$ and $h(\rho)\coloneqq-\log\det\rho$ be the logarithmic barrier.
Then, the function $f$ is $1$-smooth relative to $h$ on $\H_{++}^d$.
\end{lemma}

\subsection{FTRL with Self-Concordant Regularizer} \label{app:olo}

This section presents the regret bound of follow-the-regularized-leader (FTRL) with self-concordant regularizers of \citet{Tsai2023b} in a slightly general form.
An \textit{online linear optimization} problem is a multi-round game between two players, say \textsc{Learner} and \textsc{Reality}.
In the $t$-th round,
\begin{itemize}
\item first, \textsc{Learner} announces an action $x_t\in\mathcal{X}$;
\item then, \textsc{Reality} reveals a loss function $f_t(x)\coloneqq\braket{v_t,x}$ for some $v_t\in\V$;
\item lastly, \textsc{Learner} suffers a loss $f_t(x_t)$.
\end{itemize}
The goal of \textsc{Learner} is to minimize the regret $\sup_{x\in\mathcal{X}}R_t(x)$, where
\[
	R_t(x) \coloneqq \sum_{\tau=1}^t f_\tau(x_\tau) -  \sum_{\tau=1}^t f_\tau(x), \quad\forall x\in\mathcal{X}.
\]
We refer readers to the lecture notes of \citet{Orabona2023a} and \citet{Hazan2016a} for a general introduction to online convex optimization.

%\subsection{FTRL with Self-Concordant Regularizers}

\begin{algorithm}[h]
\caption{FTRL for online linear optimization} 
\label{alg:ftrl}
\begin{algorithmic}[1]
\STATE $x_1 \in \argmin_{x\in\mathcal{X}} \eta_0^{-1}\varphi(x)$.
\FORALL{$t \in \mathbb{N}$}
\STATE Announce $x_{t}$ and receive $v_t \in \V$.
\STATE Compute a learning rate $\eta_t>0$.
\STATE $x_{t + 1} \leftarrow \argmin_{x\in\mathcal{X}} \braket{v_{1:t}, x } + \eta_t^{-1}\varphi(x)$.
\ENDFOR
\end{algorithmic}
\end{algorithm}

FTRL is presented in Algorithm~\ref{alg:ftrl}.
We assume that the regularizer $\varphi$ is a self-concordant function.

\begin{assumption} \label{ass:regularizer}
The function $\varphi$ is an $M$-self-concordant function such that $\mathcal{X}$ is contained in the closure of $\dom \varphi$ and $\min_{x\in\mathcal{X}}\varphi(x) = 0$.
% Suppose that $\nabla^2\varphi(x)$ is positive definite for all $x \in \mathcal{X} \cap \dom \varphi$.
The Hessian $\nabla^2\varphi(x)$ is positive definite for all $x \in \mathcal{X} \cap \dom \varphi$.
\end{assumption}

Let $\norm{\cdot}_x\coloneqq(D^2\varphi(x)[\cdot,\cdot])^{1/2}$ be the local norm associated with $\varphi$ at $x$ and $\norm{\cdot}_{x,\ast}$ be its dual norm.
The theorem below bounds the regret of Algorithm~\ref{alg:ftrl}.
\begin{theorem}[Theorem~3.2 of \citet{Tsai2023b}]\label{thm:ftrl}
Assume that Assumption~\ref{ass:regularizer} holds and $\eta_{t-1}\norm{v_t}_{x_t,\ast} \leq 1/(2M)$ for all $t\in\N$.
Then, Algorithm~\ref{alg:ftrl} satisfies
\[
	R_t(x)
	\leq \frac{\varphi(x)}{\eta_t}
	+ \sum_{\tau=1}^t \eta_{\tau-1}\norm{v_\tau}_{x_\tau,\ast}^2, \quad \forall t\in\N.
\]
\end{theorem}

\begin{remark}
It is important to notice that the regret analysis of \citet{Tsai2023b} directly extends for the quantum setup. 
\end{remark}

The following corollary has appeared in the proof of Theorem~6.2 of \citet{Tsai2023b} implicitly.
We provide the statement and the proof for completeness.

\begin{corollary} \label{cor:ftrl}
Assume that Assumption~\ref{ass:regularizer} holds.
Moreover, assume that $\norm{v_t}_{x_t,\ast}\leq G$ for all $t\in\N$.
Then, for any $D>0$, Algorithm~\ref{alg:ftrl} with
\[
	\eta_t = \frac{D}{ \sqrt{ \sum_{\tau=1}^t \norm{v_\tau}_{x_\tau,\ast}^2 + 4M^2G^2D^2 + G^2}},\quad\forall t\in\N,
\]
satisfies
\[
	R_t(x) \leq \left( \frac{\varphi(x)}{D} + 2D \right)\sqrt{ \sum_{\tau=1}^t \norm{ v_\tau }_{ x_\tau, \ast }^2 + 4M^2G^2D^2 + G^2},\quad\forall t\in\N.
\]
\end{corollary}

\begin{proof}
First, the learning rates satisfy $\eta_{t-1}\norm{v_t}_{x_t,\ast}\leq 1/(2M)$ for all $t\in\N$ because
\[
	\eta_{t-1}\norm{v_t}_{x_t,\ast} = \frac{D \norm{v_t}_{x_t,\ast}}{ \sqrt{ \sum_{\tau=1}^{t-1} \norm{v_\tau}_{x_\tau,\ast}^2 + 4M^2G^2D^2 + G^2 }} \leq \frac{DG}{\sqrt{4M^2G^2D^2}} = \frac{1}{2M}.
\]
By Theorem~\ref{thm:ftrl} and Lemma~4.13 of \citet{Orabona2023a}, we have
\[
\begin{split}
	R_t(x)
	&\leq \frac{\varphi(x)}{D} \sqrt{ \sum_{\tau=1}^t \norm{ v_\tau }_{ x_\tau, \ast }^2 + 4M^2G^2D^2 + G^2 }
	+ D\sum_{\tau=1}^t \frac{ \norm{v_\tau}_{x_\tau,\ast}^2 }{ \sqrt{\sum_{s=1}^{\tau-1} \norm{v_s}_{x_s,\ast}^2 + 4M^2G^2D^2 + G^2} } \\
	&\leq \frac{\varphi(x)}{D} \sqrt{ \sum_{\tau=1}^t \norm{ v_\tau }_{ x_\tau, \ast }^2 + 4M^2G^2D^2 + G^2 }
	+ D\sum_{\tau=1}^t \frac{ \norm{v_\tau}_{x_\tau,\ast}^2 }{ \sqrt{\sum_{s=1}^\tau \norm{v_s}_{x_s,\ast}^2} } \\
	&\leq \frac{\varphi(x)}{D} \sqrt{ \sum_{\tau=1}^t \norm{ v_\tau }_{ x_\tau, \ast }^2 + 4M^2G^2D^2 + G^2 }
	+ 2D \sqrt{\sum_{\tau=1}^t \norm{ v_\tau }_{ x_\tau, \ast }^2} \\
	&\leq \left( \frac{\varphi(x)}{D} + 2D \right)\sqrt{ \sum_{\tau=1}^t \norm{ v_\tau }_{ x_\tau, \ast }^2 + 4M^2G^2D^2 + G^2}.
\end{split}
\]
This completes the proof.
\end{proof}

\subsection{Online-to-Batch Conversion}
This section recaps the online-to-batch conversion proposed by \citet{Cesa-Bianchi2004a}.
Consider the following optimizaiton problem:
\[
	\min_{x\in\mathcal{X}} f(x),
\]
with a stochastic first-order oracle $\O$ that returns an unbiased estimate $\O(x)\in\V$ of $\nabla f (x)$ given any $x \in \mathcal{X}$.
Algorithm~\ref{alg:online_to_batch} presents the online-to-batch conversion and Theorem~\ref{thm:online_to_batch} presents its theoretical guarantee.

\begin{algorithm}[h]
\caption{Online-to-batch conversion} 
\label{alg:online_to_batch}
\hspace*{\algorithmicindent} \textbf{Input: } An online learning algorithm $\mathcal{A}$.
\begin{algorithmic}[1]
\STATE Get $x_1$ from $\mathcal{A}$.
\FORALL{$t \in \mathbb{N}$}
\STATE Output $\bar{x}_t\coloneqq (1/t)x_{1:t}$.
\STATE $g_t = \O(x_t)$.
\STATE Send $f_t(x)\coloneqq\braket{g_t,x}$ to $\mathcal{A}$.
\STATE Get $x_{t+1}$ from $\mathcal{A}$.
\ENDFOR
\end{algorithmic}
\end{algorithm}

\begin{theorem} \label{thm:online_to_batch}
Let $R_t(x)\coloneqq \sum_{\tau=1}^t \braket{g_\tau, x_\tau - x}$ be the regret of the online algorithm $\mathcal{A}$ against $x\in\mathcal{X}$.
Assume that the stochastic gradients are unbiased, i.e., $\E[g_t|g_1,\ldots,g_{t-1},x_1,\ldots,x_{t}] = \nabla f(x_t)$.
Then, for any $x\in\mathcal{X}$, Algorithm~\ref{alg:online_to_batch} satisfies
\[
	\E\left[ f(\bar{x}_t) - f(x) \right] \leq \frac{\E[ R_t(x) ]}{t}, \quad\forall t\in\N,
\]
and
\[
	\E\left[ \sum_{\tau=1}^t (f(\bar{x}_\tau) - f(x)) \right] \leq (1 + \log t) \max_{1\leq\tau\leq t}\E[ R_\tau(x) ], \quad\forall t\in\N,
\]
where the expectation is taken with respect to the stochastic gradients $\{ g_t \}$.
\end{theorem}

In Theorem~\ref{thm:online_to_batch}, the first inequality can be found in Theorem~3.1 of \citet{Orabona2023a}, and the second inequality follows immediately by summing the first one over $t$.

% !TEX root = ../paper.tex
\ifdefined\arxiv
\section{Proofs}
\else
\section{PROOFS}
\fi

\subsection{Local Norm and Dual Local Norm} \label{app:local_norm}

\begin{lemma}
The local norm $\norm{\cdot}_{\rho}$ is given by $\norm{X}_{\rho} = \sqrt{ \tr((\rho^{-1}X)^2) } = \sqrt{ \tr((\rho^{-1/2}X\rho^{-1/2})^2) }$.
\end{lemma}

\begin{proof}
By Appendix~A.4.1 of \citet{Boyd2004a} and Example~3.20 of \citet{Hiai2014a}, we write
\[
\begin{split}
	\norm{X}_{\rho}^2 &= D^2 h(\rho)[X,X] \\
	&= \frac{\du^2}{\du t^2} -\log\det(\rho + tX)\bigg\rvert_{t=0} \\
	&= \frac{\du}{\du t} -\tr((\rho + tX)^{-1}X) \bigg\rvert_{t=0} \\
	&= \tr\left( -\frac{\du}{\du t}(\rho+tX)^{-1}\bigg\rvert_{t=0} X \right) \\
	&= \tr( \rho^{-1}X\rho^{-1}X).
\end{split}
\]
The lemma follows.
\end{proof}

\begin{lemma}\label{lem:dual_norm}
The dual norm of $\norm{\cdot}_{\rho}$ is given by $\norm{X}_{\rho,\ast} = \sqrt{ \tr( (\rho X)^2 ) } = \sqrt{ \tr((\rho^{1/2}X\rho^{1/2})^2) }$. 
\end{lemma}
\begin{proof}
By the definition of dual norm,
\begin{equation*}
	\norm{X}_{\rho,\ast} = \sup_{\norm{\tau}_{\rho}=1} \abs{ \tr(X\tau) }
	= \sup_{\norm{\tau}_{\rho}=1} \abs{ \tr(\rho^{1/2}X\rho^{1/2}\rho^{-1/2}\tau\rho^{-1/2}) }.
\end{equation*}
By the Cauchy-Schwarz inequality and $\norm{\tau}_{\rho}=1$, we have
\[
\abs{ \tr(\rho^{1/2}X\rho^{1/2}\rho^{-1/2}\tau\rho^{-1/2}) }
\leq \sqrt{ \tr((\rho^{1/2}X\rho^{1/2})^2) \tr((\rho^{-1/2}\tau\rho^{-1/2})^2) }
= \sqrt{ \tr((\rho^{1/2}X\rho^{1/2})^2) }.
\]
Then,
\begin{equation*}
	\sup_{\norm{\tau}_{\rho}=1} \abs{ \tr(X\tau) } \leq \sup_{\norm{\tau}_{\rho}=1}\sqrt{ \tr((\rho^{1/2}X\rho^{1/2})^2) }
	= \sqrt{ \tr((\rho^{1/2}X\rho^{1/2})^2) }.
\end{equation*}
The equality can be achieved by taking
\[
	\tau = \frac{\rho X\rho}{ \sqrt{ \tr((\rho^{1/2}X\rho^{1/2})^2) } }.
\]
The lemma follows.
\end{proof}

\subsection{Proof of Lemma~\ref{lem:lipschitz}} \label{app:lipschitz}
We write
\begin{equation*}
	\norm{\nabla f(\rho)}_{\rho,\ast}^2
	= \tr\left(\left( \E \frac{\rho^{1/2}A\rho^{1/2}}{\tr(\rho A)} \right)^{\!\!2}\right)
	\leq \E \tr\left(\left( \frac{\rho^{1/2}A\rho^{1/2}}{\tr(\rho A)} \right)^{\!\!2}\right)
	= \E \frac{\tr((\rho^{1/2} A \rho^{1/2})^2)}{(\tr(\rho^{1/2} A \rho^{1/2}))^2}
	\leq 1,
\end{equation*}
where the first inequality follows from the convexity of $\tr(A^2)$ and Jensen's inequality, and the second inequality follows from the inequality $0\leq \tr(A^2) \leq (\tr A)^2$ for $A\in\H_+^d$.

\subsection{Proof of Lemma~\ref{lem:self_bounding}} \label{app:self_bounding}
The first few steps follow from Lemma~4.7 of \citet{Tsai2023b}.
The main simplification of the original proof is the use of Theorem~\ref{thm:dikin_ellipsoid}, which we will see later.

Write $\alpha_\rho = \alpha_\rho(\nabla f(\rho))$ for simplicity and assume $\norm{ \nabla f(\rho) + \alpha_\rho I}_{\rho,\ast}\neq 0$.
Otherwise, the lemma holds immediately.
Fix $\rho\in\ri\D_d$.
By relative smoothness of $f$ (Definition~\ref{def:relative_smoothness} and Lemma~\ref{lem:relative_smoothness}), we have
\[
	f(\rho') \leq f(\rho) + \braket{ \nabla f(\rho) , \rho' - \rho } + \left[ h(\rho') -  h(\rho) - \braket{ \nabla h(\rho), \rho' - \rho } \right],\quad\forall \rho'\in\ri\D_d,
\]
where $\braket{U,V} \coloneqq \tr(U^\ast V)$ is the Hilbert-Schmidt inner product on $\H^d$.
Then, by self-concordance of $h$ (Theorem~\ref{thm:self_concordant_smoothness} and Lemma~\ref{lem:self_concordance}), we have
\[
	f(\rho') \leq f(\rho) + \braket{ \nabla f(\rho) , \rho' - \rho } + \norm{ \rho' - \rho }_{\rho}^2, \quad \forall \rho'\in\ri\D_d: \norm{ \rho' - \rho }_{\rho}\leq 1/2,
\]
where $\norm{\cdot}_{\rho}$ is the local norm associated with $h$.
%Since $\norm{ \rho' - \rho }_{\rho}\leq 1/2$, by Theorem~\ref{thm:dikin_ellipsoid}, we have $\rho'\in\dom h = \H_{++}^d$.
%Therefore, we can rewrite the conditions as $\norm{\rho'-\rho}_{\rho}\leq 1/2$ and $\tr\rho'=1$.
Since $\braket{I,\rho-\rho'} = 0$, we write
\[
	f(\rho') \leq f(\rho) + \braket{ \nabla f(\rho) + \alpha_\rho I, \rho' - \rho } + \norm{ \rho' - \rho }_{\rho}^2, \quad \forall \rho'\in\ri\D_d:  \norm{ \rho' - \rho }_{\rho}\leq 1/2.
\]
Rearraging the terms and taking supremum over all possible $\rho'$, we obtain
\begin{equation} \label{eq:self_bounding_first}
	\sup_{\rho'\in\ri\D_d:\norm{\rho'-\rho}_{\rho}\leq 1/2} \braket{ -\nabla f(\rho) - \alpha_\rho I, \rho'-\rho } - \norm{\rho' - \rho}_{\rho}^2 \leq f(\rho) - \min_{\rho'\in\D_d} f(\rho').
\end{equation}

Next, for $\beta\in [0,1/2]$, define
\begin{equation*}
	\rho'_\beta \coloneqq \rho -\beta \frac{ \rho (\nabla f(\rho) + \alpha_\rho I) \rho}{ \norm{ \nabla f(\rho) + \alpha_\rho I}_{\rho,\ast} }.
\end{equation*}
We will plug $\rho'_\beta$ into the supremum \eqref{eq:self_bounding_first} and must verify that $\rho'_\beta$ satisfies the constraints.
Since
\begin{multline*}
	\tr\left( \rho (\nabla f(\rho) + \alpha_\rho I) \rho \right)
	\ifdefined\arxiv \\ \fi
	= \tr\left(\rho\left(-\E\frac{A}{\tr(A\rho)} + \E\frac{\tr(A\rho^2)}{ \tr(A\rho)\tr(\rho^2)}I \right)\rho\right)
	= \E\left[ -\frac{\tr(A\rho^2)}{\tr(A\rho)} + \frac{\tr(A\rho^2) \tr(\rho^2) }{\tr(A\rho) \tr(\rho^2) } \right]
	= 0,
\end{multline*}
we have $\tr\rho_\beta' = \tr\rho = 1$.
Second, by the definition of $\alpha_\rho$ \eqref{eq:alpha} and Lemma~\ref{lem:lipschitz}, we have $\norm{\rho'_\beta - \rho}_{\rho} = \beta  \leq 1/2$.
At last, we need to verify $\rho_\beta'> 0$.
Since $\norm{ \rho'_\beta - \rho }_{\rho}\leq 1/2$, by Theorem~\ref{thm:dikin_ellipsoid}, we have $\smash{\rho'_\beta\in\dom h = \H_{++}^d}$ and $\rho'_\beta>0$.
The application of Theorem~\ref{thm:dikin_ellipsoid} simplifies the proof of \citet{Tsai2023b} because we no longer need to check $\rho_\beta' > 0$ explicitly.
Plugging $\rho'_\beta$ and lower bounding the supremum \eqref{eq:self_bounding_first}, we have
\begin{equation*}
	\sup_{0\leq\beta\leq 1/2}(\beta - \beta^2)\norm{\nabla f(\rho) + \alpha_\rho I}_{\rho,\ast}^2 \leq f(\rho) - \min_{\rho'\in\D_d} f(\rho').
\end{equation*}
The lemma follows by noticing that $\beta=1/2$ achieves the supremum.

\subsection{Proof of Lemma~\ref{lem:oracle_b}} \label{app:oracle}
First, it is clear that the unbiasedness property holds.
By Lemma~\ref{lem:lipschitz}, we can take $G=1$.
For the variance, we write
\[
\begin{split}
	\E\left[ \norm{g_t - \nabla f(\rho_t)}_{\rho_t,\ast}^2 | \mathcal{H}_t \right]
	&= \E\left[ \left\lVert \frac{1}{B}\sum_{b=1}^B (\nabla \ell_b(\rho_t)- \nabla f(\rho_t)) \right\rVert_{\rho_t,\ast}^2 \bigg| \mathcal{H}_t \right] \\
	&= \frac{1}{B^2}\E\left[ 
	\sum_{1\leq b,b'\leq B}\tr( (\nabla \ell_b(\rho_t)- \nabla f(\rho_t))\rho_t(\nabla \ell_{b'}(\rho_t)- \nabla f(\rho_t))\rho_t ) \bigg| \mathcal{H}_t \right] \\
	&= \frac{1}{B^2} \E\left[ \sum_{b=1}^B \norm{ \nabla \ell_b(\rho_t)- \nabla f(\rho_t) }_{\rho_t,\ast}^2 \bigg| \mathcal{H}_t \right],
\end{split}
\]
where the second equality follows from the explicit formula of the dual local norm (Lemma~\ref{lem:dual_norm});
the third equality follows from the independence of $b,b'$ and unbiasedness of $\nabla\ell_b$ and $\nabla\ell_{b'}$.
Finally, by the triangle inequality and Lemma~\ref{lem:lipschitz},
\begin{multline}
	\E\left[ \norm{ \nabla \ell_b(\rho_t)- \nabla f(\rho_t) }_{\rho_t,\ast}^2 \big| \mathcal{H}_t \right]
	\ifdefined\arxiv \\ \fi
	\leq \E\left[ \left( \norm{ \nabla \ell_b(\rho_t)}_{\rho_t,\ast} + \norm{\nabla f(\rho_t)}_{\rho_t,\ast} \right)^2 | \mathcal{H}_t\right]
	\leq \E\left[ (1+1)^2 | \mathcal{H}_t \right]
	\leq 4.
\end{multline}
The lemma follows.

\subsection{Proof of Theorem~\ref{thm:sda}} \label{app:sda_theorem}
Let $h(\rho)=-\log\det\rho - d\log d$ be the logarithmic barrier.
Note that Algorithm~\ref{alg:sda} is derived by applying Algorithm~\ref{alg:ftrl} with the regularizer $h$ to the online linear optimization problem (Appendix~\ref{app:olo}), followed by the online-to-batch conversion (Algorithm~\ref{alg:online_to_batch}).

Fix $\rho\in\D_d$.
To deal with the unboundedness of $h$ on $\D_d$, we first apply the technique used in Lemma~10 of \citet{Luo2018a}.
Let any $\rho\in\D_d$, define $\tilde{\rho} = (t/(t+1))\rho + 1/(t+1)(I/d) \in \D_d$.
Then, by convexity,
\begin{multline*}
	f(\tilde{\rho}) - f(\rho) \leq \braket{ \nabla f(\tilde{\rho}), \tilde{\rho} - \rho }
	\ifdefined\arxiv \\ \fi
	= \E\left[ \frac{\tr(A\rho) - \tr(A\tilde{\rho})}{\tr(A\tilde{\rho})} \right]
	= \E\left[ \frac{(1+1/t)\tr(A\tilde{\rho}) - (1/dt)\tr A - \tr(A\tilde{\rho})}{\tr(A\tilde{\rho})} \right]
	\leq \frac{1}{t}.
\end{multline*}
Therefore,
\[
	\E\left[ f(\bar{\rho}_t) - f(\rho) \right] \leq \E\left[ f(\bar{\rho}_t) - f(\tilde{\rho}) \right] + \frac{1}{t}.
\]
Applying the first inequality in Theorem~\ref{thm:online_to_batch}, we have
\begin{equation} \label{eq:online_to_batch}
	\E\left[ f(\bar{\rho}_t) - f(\rho) \right] \leq \frac{\E[ R_t(\tilde{\rho}) ] + 1}{t}
	\leq \frac{\max_{1\leq\tau\leq t} \E R_\tau(\tilde{\rho}) + 1}{t},
\end{equation}
where
\[
	R_t(\tilde{\rho}) = \sum_{\tau=1}^t \braket{ g_\tau , \rho_\tau - \tilde{\rho} } = \sum_{\tau=1}^t \braket{ g_\tau + \alpha_{\rho_\tau}(g_\tau) I, \rho_\tau - \tilde{\rho} }
\]
is the regret of Algorithm~\ref{alg:ftrl} applied to the online linear optimization problem with $v_t = g_t + \alpha_{\rho_\tau}(g_\tau) I$ (see Appendix~\ref{app:olo}).

Next, by assumption and the definition of $\alpha_{\rho}$ \eqref{eq:alpha}, we have $\norm{ g_\tau + \alpha_{\rho_\tau}(g_\tau) I }_{ \rho_\tau, \ast } \leq \norm{ g_\tau }_{ \rho_\tau, \ast } \leq G$.
Applying Corollary~\ref{cor:ftrl} with $M=1$ and $D=\sqrt{d}$, we obtain
\[
	R_t(\tilde{\rho})
	\leq\left( \frac{h(\tilde{\rho})}{\sqrt{d}} + 2\sqrt{d} \right)\sqrt{ G_t + 4dG^2 + G^2}
	\leq (\log t + 3)\sqrt{ dG_t + 4d^2G^2 + dG^2},
\]
where $G_t = \sum_{\tau=1}^t \norm{ g_\tau + \alpha_{\rho_\tau}(g_\tau) I }_{ \rho_\tau, \ast }^2$.
The last inequality follows from $h(\tilde{\rho})\leq d\log(t+1) \leq d\log t + d$.
By Jensen's inequality, the expected regret is bounded by
\[
	\E R_t(\tilde{\rho}) \leq (\log t + 3)\sqrt{d \E G_t + 4d^2G^2 + dG^2}.
\]
Then, we take maximum over $t$ on both sides.
Since the upper bound is increasing in $t$, we obtain
\begin{equation} \label{eq:regret_bound}
	\max_{1\leq\tau\leq t}\E R_\tau(\tilde{\rho})
	\leq (\log t + 3)\sqrt{d \E G_t + 4d^2G^2 + dG^2}.
\end{equation}

Now, we upper bound $\E G_t$ by $\max_{1\leq\tau\leq t}\E R_{\tau}(\tilde{\rho})$.
Denote by $\mathcal{H}_t = \{ g_1,\ldots, g_{t-1}, \rho_1,\ldots,\rho_t\}$ the past information before obtaining $g_t$.
By the linearity of $\alpha_\rho$ \eqref{eq:alpha}, we have
\[
	\E [ \alpha_{\rho_t}(g_t) | \mathcal{H}_\tau ] = \alpha_{\rho_t} (\nabla f(\rho_t)).
\]
Recall from Lemma~\ref{lem:dual_norm} that $\norm{X}_{\rho,\ast} = \norm{\rho^{1/2}X\rho^{1/2}}_F$, where $\norm{X}_F\coloneqq\sqrt{\tr(X^\ast X)}$ is the Frobenius norm.
By the law of total expectation, Lemma~\ref{lem:dual_norm}, and the variance decomposition $\E[\norm{X}_F^2] = \E[\norm{X-\E X}_F^2] + \norm{\E X}_F^2$ for a random matrix $X$, we have
\begin{equation} \label{eq:variance}
\begin{split}
	&\E \left[ \norm{g_\tau + \alpha_{\rho_\tau}(g_\tau) I }_{\rho_\tau,\ast}^2 \right] \\
	= &
	\E_{\mathcal{H}_\tau}\E \left[ \norm{g_\tau + \alpha_{\rho_\tau}(g_\tau) I }_{\rho_\tau,\ast}^2 | \mathcal{H}_\tau \right]
	\\
	= &
	\E_{\mathcal{H}_\tau}\E \left[ \norm{g_\tau - \nabla f(\rho_\tau) + \alpha_{\rho_\tau}(g_\tau - \nabla f(\rho_\tau))I }_{\rho_\tau,\ast}^2 | \mathcal{H}_\tau \right]
	+ \E_{\mathcal{H}_\tau} \left[ \norm{\nabla f(\rho_\tau) + \alpha_{\rho_\tau}( \nabla f(\rho_\tau) )I }_{\rho_\tau,\ast}^2 \right]
	\\
	= &\E_{\mathcal{H}_\tau}\E \left[ \norm{g_\tau - \nabla f(\rho_\tau) + \alpha_{\rho_\tau}(g_\tau - \nabla f(\rho_\tau))I }_{\rho_\tau,\ast}^2 | \mathcal{H}_\tau \right]
	+ \E \left[ \norm{\nabla f(\rho_\tau) + \alpha_{\rho_\tau}( \nabla f(\rho_\tau) )I }_{\rho_\tau,\ast}^2 \right].
	%&\qquad \leq \E \norm{g_\tau + \alpha_{\rho_\tau}(g_\tau) I - \nabla f(\rho_\tau) - \alpha_{\rho_\tau}(\nabla f(\rho_\tau))I }_{\rho_\tau,\ast}^2
	%+ \E\left[ \norm{\nabla f(\rho_\tau) + \alpha_{\rho_\tau}( \nabla f(\rho_\tau) )I }_{\rho_\tau,\ast}^2 \right]
\end{split}
\end{equation}
%where the expectation is taken with respect to $g_\tau$ conditional on $g_1,\ldots,g_{\tau-1},\rho_1,\ldots,\rho_\tau$.
We bound the two terms separately.
By the definition of $\alpha_\rho$ \eqref{eq:alpha} and the bounded-variance assumption,
\[
	\E_{\mathcal{H}_\tau} \E \left[ \norm{g_\tau - \nabla f(\rho_\tau) + \alpha_{\rho_\tau}(g_\tau - \nabla f(\rho_\tau))I  }_{\rho_\tau,\ast}^2 | \mathcal{H}_\tau \right]
	\leq \E_{\mathcal{H}_\tau} \E \left[ \norm{g_\tau - \nabla f(\rho_\tau)}_{\rho_\tau, \ast}^2 | \mathcal{H}_\tau \right]
	\leq \sigma^2.
\]
Furthermore, let $\delta_\tau \coloneqq \E f(\rho_\tau) - f(\tilde{\rho})$.
By the self-bounding-type property (Lemma~\ref{lem:self_bounding}), we have
\[
	\E\left[ \norm{\nabla f(\rho_\tau) + \alpha_{\rho_\tau}( \nabla f(\rho_\tau) )I }_{\rho_\tau,\ast}^2 \right]
	\leq \E f(\rho_\tau) - \min_{\rho\in\D_d} f(\rho) \leq \delta_\tau.
\]
Therefore, we have $\eqref{eq:variance} \leq \sigma^2 + 4\delta_\tau$ and 
\[
	\E G_t = \sum_{\tau=1}^t \E \left[ \norm{g_\tau + \alpha_{\rho_\tau}(g_\tau) I }_{\rho_\tau,\ast}^2 \right]
	\leq \sigma^2t + 4 \sum_{\tau=1}^t\delta_\tau.
\]
By the second inequality in Theorem~\ref{thm:online_to_batch}, we have $\smash{\sum_{\tau=1}^t\delta_\tau} \leq (1+\log t)\max_{1\leq\tau\leq t}\E R_\tau(\tilde{\rho})$.
Hence,
\[
	\E G_t \leq \sigma^2 t + 4(1+\log t)\max_{1\leq\tau\leq t}\E R_\tau(\tilde{\rho}).
\]

Finally, combining with the bound on the expected regret \eqref{eq:regret_bound}, we have
\[
	\max_{1\leq \tau \leq t}\E R_{\tau}(\tilde{\rho}) \leq (3+\log t)\sqrt{4d(1+\log t)\max_{1\leq\tau\leq t}\E R_{\tau}(\tilde{\rho}) + 4d^2G^2 + dG^2 + \sigma^2 dt}.
\]
By Lemma~4.24 of \citet{Orabona2023a}, solving for $\max_{1\leq \tau \leq t}\E R_{\tau}(\tilde{\rho})$ gives 
\[
	\max_{1\leq \tau \leq t}\E R_{\tau}(\tilde{\rho}) \leq 4d(3+\log t)^3 + 2(3+\log t)\sqrt{\sigma^2 dt + 4d^2G^2 + dG^2}.
\]
The theorem follows by combining the above inequality with \eqref{eq:online_to_batch}.
% !TEX root = ../paper.tex
\ifdefined\arxiv
\section{Additional Nuumerical Results}
\else
\section{ADDITIONAL NUMERICAL RESULTS}
\fi
\label{app:add_numerical_results}

\begin{figure}[b!]
\caption{Performances of all algorithms in Table~\ref{tab:time_complexity_comparison}, SPDHG, and EMD with line search for solving 20 randomly generated Poisson inverse problem instances.
For each algorithm, the solid line represents its average error, and the shaded region indicates the 95\% confidence interval.}
\label{fig:pip_full}
\centering

\begin{subfigure}{0.48\columnwidth}
	\centering
	\caption{The true signal.}\includegraphics[width=\columnwidth]{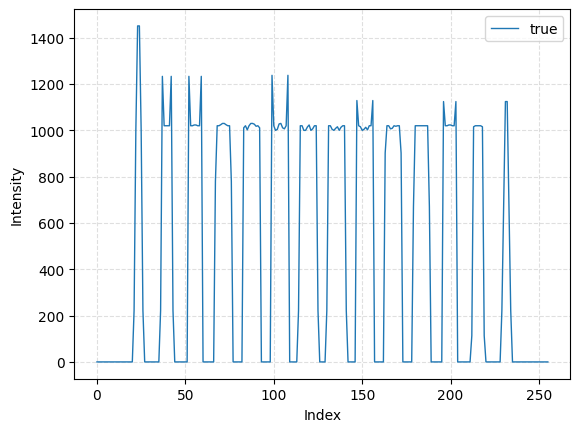}
\end{subfigure}

\begin{subfigure}{0.48\columnwidth}
	\centering
	\caption{Approximate optimization error versus the number of epochs.}
	\includegraphics[width=\columnwidth]{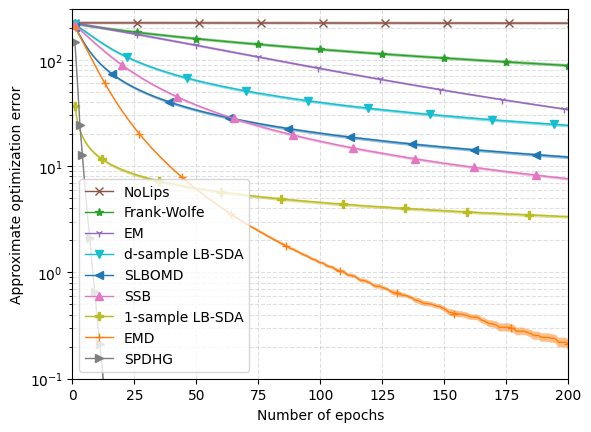}
\end{subfigure}%
\hfill
\begin{subfigure}{0.48\columnwidth}
	\centering
	\caption{Normalized estimation error versus the number of epochs.}
	\includegraphics[width=\columnwidth]{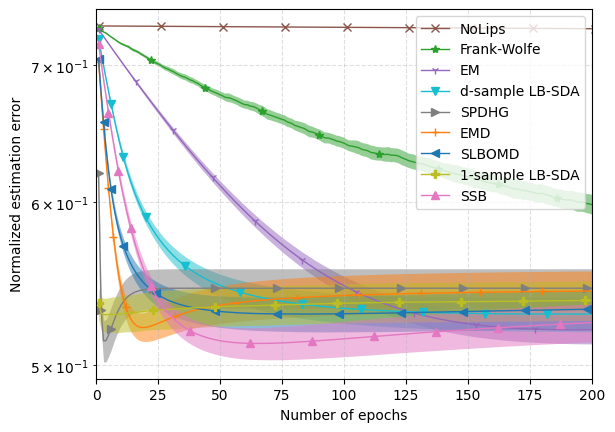}
\end{subfigure}%
\end{figure}

This section presents additional numerical results in terms of the number of epochs, where one \textit{epoch} refers to one full pass of the dataset.
Specifically, one epoch of $B$-sample LB-SDA corresponds to $n/B$ iterations, one epoch of other stochastic methods corresponds to $n$ iterations, and one epoch of batch methods corresponds to one iteration.
The number of epochs is proportional to the number of gradient evaluated by the algorithm.
Note that EMD and diluted iMLE compute function values in the Armijo line search procedure, which pass through the dataset for more than once.
Therefore, the numerical results in terms of the number of epochs favor these two algorithms.
Since computing the function values is much faster than computing the gradients, we still present numerical results in terms of the number of epochs.

Figure~\ref{fig:pip_full} presents results for the experiment of the Poisson inverse problem in Section~\ref{sec:numerical_pip}, while Figure~\ref{fig:mlqst_full} presents results for the experiment of ML quantum state tomography in Section~\ref{sec:numerical_mlqst}.
For the Poisson inverse problem, we randomly generated 20 problem instances under the setup described in Section~\ref{sec:numerical_pip}.
We then reported the average performance of the algorithms on them.
For ML quantum state tomography, we considered only one problem instance because the experiment already took one week.

%\added{Figure~\ref{fig:pip_random} presents the average performances of algorithms for solving 20 Poisson inverse problem instances.
%The problem instances are randomly generated under the setup described in Section~\ref{sec:numerical_pip}.
%Regarding the optimization error, 1-sample LB-SDA shows a consistent improvement over other methods with explicit complexity guarantees.
%We did not conduct multiple experiments for ML quantum state tomography because a single experiment took one week.
%}

\begin{figure}[t!]
\caption{Performances of all algorithms in Table~\ref{tab:time_complexity_comparison}, iMLE, diluted iMLE, and EMD with line search for computing the ML estimate for quantum state tomography.}
\label{fig:mlqst_full}
\centering

\begin{subfigure}{0.48\columnwidth}
	\centering
	\caption{Approximate optimization error versus the number of epochs.}
	\includegraphics[width=\columnwidth]{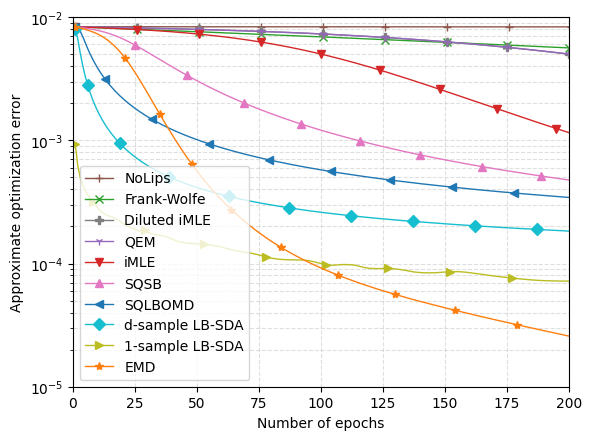}
\end{subfigure}%
\hfill
\begin{subfigure}{0.48\columnwidth}
	\centering
	\caption{Fidelity between the iterates and the $W$ state versus the number of epochs.}
	\includegraphics[width=\columnwidth]{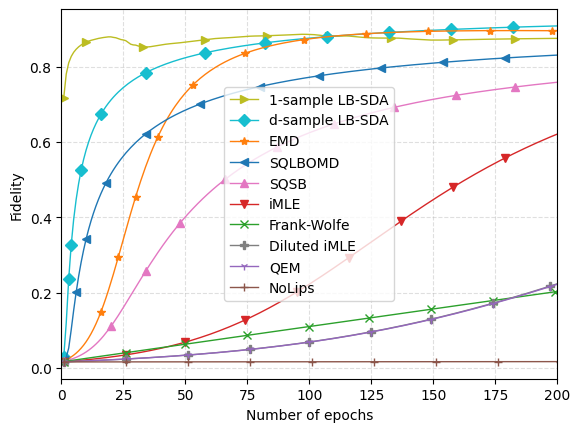}
\end{subfigure}%
\end{figure}

In terms of the optimization error, $1$-sample LB-SDA outperforms other methods with clear complexity guarantees.
Note that $1$-sample LB-SDA converges faster than $d$-sample LB-SDA.
The reason is that after the $s$-th epoch, $1$-sample LB-SDA has performed $ns$ iterations, whereas $d$-sample LB-SDA has only performed $ns/d$ iterations.
According to Corollary~\ref{cor:sda}, the optimization error bound of $1$-sample LB-SDA is $\tildeO(d/(ns) + \sqrt{d/(ns)})$, smaller than the $\smash{\tildeO(d^2/(ns) + \sqrt{d/(ns)})}$ optimization error bound of $d$-sample LB-SDA.
In terms of the fidelity, $1$-sample LB-SDA achieves the best performance among all methods with explicit complexity guarantees for computing the ML estimate for quantum state tomography.

%\begin{figure}[]
%\caption{\added{Performances of all algorithms in Table~\ref{tab:time_complexity_comparison}, SPDHG, and EMD with line search for solving 20 Poisson inverse problem instances with randomly generated measurements $\{ b_i \}$.
%For each algorithm, the solid line represents the average error, and the shaded region indicates the 95\% confidence interval.}}
%\label{fig:pip_random}
%\centering
%
%\begin{subfigure}{0.48\columnwidth}
%	\centering
%	\caption{Approximate optimization error versus the number of epochs.}
%	\includegraphics[width=\columnwidth]{./figures/pip/phantom16/epoch-error-errorbar.png}
%\end{subfigure}%
%\hfill
%\begin{subfigure}{0.48\columnwidth}
%	\centering
%	\caption{Normalized estimation error versus the number of epochs.}
%	\includegraphics[width=\columnwidth]{./figures/pip/phantom16/epoch-distance-errorbar.png}
%\end{subfigure}%
%\end{figure}

\bibliography{./refs}

\end{document}